\documentclass[reqno,12pt,letterpaper]{amsart}

\usepackage{style}

\begin{document}

\title[A global definition of quasinormal modes for Kerr--AdS Black Holes]
{A global definition of quasinormal modes for Kerr--AdS Black Holes}
\author{Oran Gannot}
\email{gannot@northwestern.edu}
\address{Department of Mathematics, Lunt Hall, Northwestern University,
Evanston, IL 60208, USA}

\begin{abstract}
The quasinormal frequencies of massive scalar fields on Kerr--AdS black holes are identified with poles of a certain meromorphic family of operators, once boundary conditions are specified at the conformal boundary. Consequently, the quasinormal frequencies form a discrete subset of the complex plane and the corresponding poles are of finite rank. This result holds for a broad class of elliptic boundary conditions, with no restrictions on the rotation speed of the black hole.            
\end{abstract}

\maketitle

\section{Introduction}
The study of quasinormal modes (QNMs) has proven useful in understanding long-time behavior of linearized perturbations throughout general relativity. These modes are solutions of the linear wave equation with harmonic time-dependence, subject to outgoing boundary conditions at event horizons.  Associated to each QNM is a complex quasinormal frequency (QNF) which determines the time evolution of a QNM: the real part describes the mode of oscillation, while the imaginary part corresponds to exponential decay or growth in time. 

The QNF spectrum depends on black hole parameters (such as cosmological constant, rotation speed, and mass), but not the precise nature of the perturbation. The distribution of QNFs in the complex plane is expected to dictate the return to equilibrium for linearized perturbations. This follows established tradition in scattering theory, where QNFs typically go by the name of scattering poles or resonances.

In particular, there has been a great deal of interest in the QNMs of asymptotically anti-de Sitter (AdS) black holes, motivated both by developments in the AdS/CFT program and by closely related questions in classical gravitation \cite{hawking:1999,konoplya:2011,witten:1998}. Understanding perturbations of such black holes is a common thread in both the physics and mathematics literature. 

According to the proposed holographic correspondence, a black hole in an AdS background is dual to a thermal state on the conformal boundary. Behavior of perturbations in the bulk therefore yields predictions on thermalization timescales for the dual gauge theory which are difficult to calculate within the strongly coupled field theory. It is also important to note that QNMs have a distinguished interpretation in the AdS/CFT correspondence \cite{cardoso:2014:jhep,horowitz:2000:prd}.

Additionally, a major unsolved problem in mathematical general relativity is the nonlinear instability of global anti-de Sitter space, in the sense that a generic perturbation of such a metric will grow and form a black hole \cite{balasubramanian2014holographic,bizon:2014:grg,bizon2015resonant,bizon:2011:prl,buchel2015conserved,craps:2014vaa,craps:2014jwa,dias:2012:cqg}. If AdS is indeed unstable, a natural question is whether the endpoint of instability is a Kerr--AdS black hole. Both of these subjects have motivated substantial interest in the nonlinear instability (or stability) of Kerr--AdS \cite{cardoso:2014:jhep,dias:2012:cqg:b,dias:2012:cqg,holzegel:2015,holzegel2013decay,holzegel:2013:cmp,holzegel2014quasimodes,holzegel:2012wt}. In particular, Holzegel--Smulevici established logarithmic decay of massive scalar fields on Kerr--AdS backgrounds \cite{holzegel2013decay} (with Dirichlet conditions imposed at the conformal boundary), and then demonstrated the optimality of this decay rate \cite{holzegel2014quasimodes} (see also \cite{gannot:2014} for the Schwarzschild--AdS case). This slow decay rate led to the conjecture that Kerr--AdS itself is nonlineary unstable.

This paper continues the study of scalar perturbations of Kerr--AdS black holes. The relevant linear equation to be solved is the Klein--Gordon equation
\begin{equation} \label{eq:kg}
\Box_g \phi  + \frac{m^2}{l^2}\phi = 0,
\end{equation}
where $l$ is related to the negative cosmological constant by $l^2 = 3/|\Lambda|$. The purpose of this paper is to provide a robust definition of QNFs for Kerr--AdS metrics which does not depend on any extra symmetries (separation of variables), and then show that the QNF spectrum forms a discrete subset of the complex plane. This means studying solutions to \eqref{eq:kg} of the form $\phi = e^{-i\lambda t^\star}u$, where $\lambda \in \mathbb{C}$ and $u$ is a stationary function, identified with its values on the time slice $\{ t^\star = 0 \}$; here $t^\star$ is a time coordinate which is regular across the event horizon. A critical observation is that the outgoing condition is equivalent to a certain smoothness requirement for $u$ at the event horizon.

Since the conformal boundary of an asymptotically AdS spacetime is timelike, there is no reason for the set of QNFs to be discrete unless \eqref{eq:kg} is augmented by boundary conditions at the conformal boundary. Choosing appropriate boundary conditions is a subtle point depending on $m^2$. When $m^2 \geq -5/4$ it suffices to rule out solutions which are not square integrable. On the other hand, when $-9/4 < m^2 < -5/4$ the problem is underdetermined and boundary conditions must be imposed.

This paper uses recent advances in the microlocal study of wave equations on black hole backgrounds due to Vasy \cite{vasy:2013} to study global Fredholm properties of the time-independent problem. Upon verifying some dynamical assumptions on the null-geodesic flow of Kerr--AdS metrics, the approach of \cite{vasy:2013} provides certain estimates for the stationary operator corresponding to \eqref{eq:kg},  at least away from the conformal boundary. Compared to recent work of Warnick \cite{warnick:2015:cmp} on QNFs of AdS black holes, there is no restriction on the rotation speed of the black hole --- see Section \ref{subsect:previous} below for more on the differences between \cite{warnick:2015:cmp} and this paper.

In Sections \ref{subsect:bessel} and \ref{subsect:adsellipticestimate}, a theory of boundary value problems for some singular elliptic operators, developed in \cite{gannot:bessel}, is reviewed. This theory applies in the Kerr--AdS setting. When the boundary conditions satisfy a type of Lopatinski{\v\i} condition for $-9/4 < m^2 < -5/4$, the results of \cite{gannot:bessel} provide elliptic estimates near the boundary --- see Section \ref{subsect:adsellipticestimate} for more details. These boundary conditions account for the majority of those considered in the physics literature \cite{avis:1978:prd,berkooz:2002:jhep,breitenlohner:1982:plb,breitenlohner:1982:ap,dias:2013:jhep,ishibashi:2004:cqg,witten2001}. This substantially generalizes the self-adjoint Dirichlet or Robin boundary conditions considered in \cite{warnick:2015:cmp}. In particular, certain time-periodic boundary conditions are admissible.

Combining estimates near the boundary with those in the interior suffices to prove the Fredholm property for the stationary operator. The inverse of this operator forms a meromorphic family, and QNFs are then defined as poles of that family. Having shown that QNFs are well defined spectral objects, a natural question is how they are distributed in the complex plane. The companion paper \cite{gannot2017} establishes the existence of QNFs converging exponentially to the real axis, generalizing the results of \cite{gannot:2014}.

A simplified discussion of Vasy's method in the slightly less involved asymptotically hyperbolic setting can found in \cite{zworski:resonances}, although the approach to proving meromorphy there differs from that of this paper.

\subsection{Main results} \label{subsect:mainresults}

For notation, the reader is referred to Section \ref{sect:kerradsmetric}. Let $\mathcal{M}_0$ denote the exterior of a Kerr--AdS spacetime with metric $g$ determined by parameters $(l, a, M)$. After modifying the original Boyer--Lindquist time slicing, there always exists an extension of $g$ across the event horizon $\mathcal{H}^+ = \{ r = r_+\}$ to a larger spacetime $\mathcal{M}_\delta$, such that the time slice $X_\delta = \{ t^\star= 0\}$ is spacelike. In the extended picture $g$ is smooth up to $\mathcal{H}_\delta = \{ r = r_+ - \delta\}$ for any sufficiently small $\delta \geq 0$.

The stationary Klein--Gordon operator $P(\lambda)$ is defined on $X_\delta$ by replacing $D_{t^\star}$ with a spectral parameter $-\lambda \in \mathbb{C}$ in the operator $r^2(\Box_g + m^2/l^2)$. Solutions of $P(\lambda)u = 0$ correspond to solutions $\phi = e^{-i\lambda t^\star}u$ of \eqref{eq:kg}. Even if one is only interested in $P(\lambda)$ acting on $X_0$, it is technically important to consider its extension to $X_\delta$. 

The prefactor $r^2$ in the definition of $P(\lambda)$ appears naturally when formulating energy identities for \eqref{eq:kg} (see \cite[Lemma 4.1.1]{warnick:2013:cmp},  \cite[Section 4]{holzegel:2012wt}, \cite[Sections 2, 3]{warnick:2015:cmp}), and does not affect solutions to the homogeneous Klein--Gordon equation; it could be replaced by any strictly positive function growing like $r^2$ at infinity. Finite energy solutions $e^{-i\lambda t^\star}u$ (as measured by the energy-momentum tensor) satisfy
\begin{equation} \label{eq:L2space}
\int_{X_\delta} |u|^2 \, r^{-1}\, dS_t < \infty,
\end{equation}
where $dS_t$ is the induced measure on $X_\delta$. The space $\mathcal{L}^2(X_\delta)$ of square integrable functions is defined with respect to the rescaled measure $r^{-1}\,dS_t$. Alternatively, the notation $\Sob^0(X_\delta) = \mathcal{L}^2(X_\delta)$ will be used at times.

The mass $m$ is required to satisfy the Breintenlohner--Freedman bound 
\[
m^2 > -9/4.
\] This restriction has a variety of consequences for the study of massive waves on asymptotically AdS spaces; in this paper, the bound must be satisfied in order to apply the results of \cite{gannot:bessel} on certain singular elliptic boundary value problems. An important related quantity is the effective mass $\nu >0$ defined by $\nu^2 = m^2 + 9/4$.

In order to define the (stationary) energy space, observe that the conformal multiple $r^{-2} g$ extends smoothly up to $\scri = \{ r^{-1} = 0\}$. Then $\mathcal{M}_\delta$ can be viewed as the interior of a manifold $\OL{\mathcal{M}}_\delta$ with two boundary components, 
\[
\partial \mathcal{M}_\delta = \scri \cup \mathcal{H}_\delta.
\] 
The set $\{ t^\star = 0 \}$ within $\OL{\mathcal{M}}_\delta$ defines a compact spacelike (with respect to $r^{-2} g$) hypersurface $\OL{X}_\delta$ with interior $X_\delta$ and boundary $\partial X_\delta = H_\delta \cup Y$, where
\[
H_\delta = \mathcal{H}_\delta \cap \OL{X}, \quad Y = \scri \cap \OL{X}.
\]
Given $\nu >0$, let $\Sob^1(X_\delta)$ denote the space of all $u \in \Sob^0(X_\delta)$ such that the conjugated derivative $r^{\nu-3/2}d(r^{3/2-\nu} u)$ lies in $\Sob^0(X_\delta)$, where the magnitude of a covector is measured by a smooth inner product on $T^*\OL{X}_\delta$ (by compactness of $\OL{X}_\delta$ this does not depend on choices).

With a view towards energy estimates, the twisted Sobolev space $\Sob^1(X_\delta)$ was introduced in \cite{warnick:2013:cmp} to define a finite energy for ``Neumann'' boundary conditions (in the sense of \cite[Section 1]{warnick:2013:cmp}), extending work of Breitenlohner--Freedman \cite{breitenlohner:1982:plb,breitenlohner:1982:ap}. See also \cite[Section 3]{warnick:2013:cmp} for additional motivation. In the elliptic setting, boundary value problems on twisted Sobolev spaces were studied in \cite{gannot:bessel}. 

Spaces with higher regularity are defined as follows: given $s=0,1$, let $\Sob^{s,k}(X_\delta)$ denote the space of all $u \in \Sob^{s}(X_\delta)$ such that $V_1 \cdots V_N u \in \Sob^{s}(X_\delta)$, where $V_1,\ldots,V_N$ is any collection of at most $k$ vector fields on $\OL{X}_\delta$ which are tangent to $Y$. Finally, set
\[
\mathcal{X}^k(X_\delta) = \{ u \in \Sob^{1,k}(X_\delta): P(0)u \in \Sob^{0,k}(X_\delta) \}.
\]
All the results in this paper require that $\mathcal{H}^+$ is a nonextremal horizon, meaning that the surface gravity $\varkappa$ associated to the horizon is positive. Explicitly,
\[
\varkappa = \frac{\Delta_r'(r_+)}{2(1-\alpha)(r_+^2+a^2)}.
\] 
The first result, valid for $\nu \geq 1$, allows for the definition of QNFs; it is stated for $P(\lambda)$ acting on the exterior time slice $X_0$.  
\begin{theo} \label{theo:maintheo1}
	If $\nu \geq 1$ and $k \in \mathbb{N}$, then 
	\[
	P(\lambda) : \mathcal{X}^k(X_0) \rightarrow \Sob^{0,k}(X_0)
	\]
	is Fredholm for $\lambda$ in the half-plane $\{ \Im \lambda > - \varkappa (k+1/2)\}$. Furthermore, given any angular sector $\Lambda \subseteq \mathbb{C}$ in the upper half-plane, there exists $R>0$ such that $P(\lambda)$ is invertible for $\lambda \in \Lambda$ and $|\lambda| > R$. 
\end{theo}

\noindent By analytic Fredholm theory, the family $\lambda \mapsto P(\lambda)^{-1}$ is meromorphic. QNFs in the half-plane $\{\Im \lambda > -\varkappa(k+1/2)\}$ are defined as poles of $P(\lambda)^{-1}: \Sob^{0,k}(X_0) \rightarrow \mathcal{X}^k(X_0)$. These poles are discrete and the corresponding residues are finite rank operators. QNMs are then elements of the finite dimensional space $\ker P(\lambda)|_{\mathcal{X}^k(X_0)}$.

Furthermore, any QNM $u\in \mathcal{X}^k(X_0)$ is smooth up to $H_0$, provided the threshold condition $\Im \lambda > -\varkappa(k+1/2)$ is satisfied; this is demonstrated during the proof of Theorem \ref{theo:maintheo1} in Section \ref{subsect:passing}. In particular, if $k' \geq k$, then the poles of 
\[
P(\lambda)^{-1}|_{\Sob^{0,k}(X_0)}, \quad P(\lambda)^{-1}|_{\Sob^{0,k'}(X_0)}
\] 
in $\{\Im \lambda > -\varkappa(k+1/2)\}$ coincide, and at regular points $P(\lambda)^{-1}|_{\Sob^{0,k}(X_0)}$ is the extension by continuity of $P(\lambda)^{-1}|_{\Sob^{0,k'}(X_0)}$. In this sense the QNF spectrum is a well defined subset of $\mathbb{C}$.  Finally, QNMs have conormal asymptotic expansions at $Y$ \cite[Proposition 4.17]{gannot:bessel}.

The analogous statement when $0 < \nu < 1$ is more involved since boundary conditions (in the sense of Bessel operators, see Section \ref{subsubsect:boundaryoperators}) must be imposed at the conformal boundary $Y$ to obtain a Fredholm problem. Fix a weighted trace $T(\lambda)$ whose ``principal part'' is independent of $\lambda$ and let
\[
\mathscr{P}(\lambda) = \begin{pmatrix} P(\lambda) \\ T(\lambda) \end{pmatrix}.
\]
The trace operator $T(\lambda)$ has an ``order'' $\mu$ (which depends on $\nu$) such that a priori 
\[
T(\lambda) : \mathcal{X}^k(X_0) \rightarrow H^{k+1-\mu}(Y)
\]
is bounded. The operator $\mathscr{P}(\lambda)$ is required to satisfy the parameter-dependent Lopatinski{\v{\i}} condition (again in the sense of Bessel operators, see Section \ref{subsubsect:lopatinskii}) with respect to an angular sector $\Lambda \subseteq \mathbb{C}$ in the upper half-plane.

\begin{theo} \label{theo:maintheo2}
If $0 < \nu < 1$ and $k \in \mathbb{N}$, then 
\[
\mathscr{P}(\lambda): \{u \in \mathcal{X}^k(X_0): T(0)u \in H^{k+2-\mu}(Y) \} \rightarrow \Sob^{0,k}(X_0) \times H^{k+2-\mu}(Y)
\] 
is Fredholm for $\lambda$ in the half-plane $\{ \Im \lambda > -\varkappa(k+1/2)\}$. Furthermore, given any angular sector $\Lambda \subseteq \mathbb{C}$ in the upper half-plane with respect to which $\mathscr{P}(\lambda)$ is parameter-elliptic, there exists $R>0$ such that $\mathscr{P}(\lambda)$ is invertible for $\lambda \in \Lambda$ and $|\lambda| > R$.
\end{theo}

\noindent QNFs in the half-plane $\{ \Im \lambda > -\varkappa(k+1/2)\}$ are again defined as poles of the meromorphic family $\lambda \mapsto \mathscr{P}(\lambda)^{-1}$. The observations following Theorem \ref{theo:maintheo1} are also applicable.

The importance of considering an extended spacetime $\mathcal{M}_\delta$ is that Theorems \ref{theo:maintheo1}, \ref{theo:maintheo2} are established by first demonstrating their validity on $X_\delta$, with $\delta >0$ strictly positive:

\begin{theo} \label{theo:mainextended}
	Theorems \ref{theo:maintheo1}, \ref{theo:maintheo2} hold true if $X_0$ is replaced by $X_\delta$, where $\delta>0$ is sufficiently small.
\end{theo}
  
\noindent In light of Theorem \ref{theo:mainextended}, it is natural to consider the relationship between the QNF spectrum (defined here as the poles of $P(\lambda)^{-1}$ or $\mathscr{P}(\lambda)^{-1}$ acting on $X_0$) and the poles of the extended inverses acting on $X_\delta$ with $\delta>0$. One implication is clear: QNFs are contained in the set of poles of the extended inverse, since surjectivity on $X_\delta$ implies surjectivity on $X_0$ and the index of both operators is zero. 

The answer to the converse question was suggested to the author by Peter Hintz; unlike the other results of this paper, it strongly uses axisymmetry of the exact Kerr--AdS metric to reduce to the case of \cite[Lemma 2.2]{hintz2015asymptotics}. To begin, define the axisymmetric distributions 
\[
\mathcal{D}'_m(X_\delta) = \{ u \in \mathcal{D}'(X_\delta) : (D_\phi - m)u = 0 \}
\]
for $m \in \mathbb{Z}$, each of which is invariant under $P(\lambda)$.

\begin{theo} \label{theo:maintheo3}
Fix $\delta > 0$. If $\nu \geq 1$ and $\Im \lambda > -\varkappa(k+1/2)$, then for each $m \in \mathbb{Z}$ the restriction map 
\[
\ker P(\lambda)|_{\mathcal{X}^k(X_\delta) \cap \mathcal{D}'_m(X_\delta)}\rightarrow \ker P(\lambda)|_{\mathcal{X}^k(X_0) \cap \mathcal{D}'_m(X_0)}
\] 
given by $u \mapsto u|_{X_0}$ is a bijection. The same is true for $0 < \nu < 1$ when $P(\lambda)$ is replaced with $\mathscr{P}(\lambda)$, provided $T(\lambda)$ is axisymmetric in the sense that $T(\lambda)\circ D_\phi= D_\phi \circ T(\lambda)$.
\end{theo}


%
%
%

\subsection{Relation to previous works} \label{subsect:previous}

The mathematical study of QNMs for AdS black holes began slightly later than for their nonnegative cosmological constant counterparts. QNMs of Schwarzschild black holes were rigorously studied by Bachelot \cite{bachelot:1991} and Bachelot--Motet-Bachelot \cite{bachelot:1993}. Meromorphy of the scattering resolvent for Schwarzschild--de Sitter black holes was established by S\'{a} Barreto--Zworski \cite{barreto:1997}, who also described the lattice structure of QNFs. Expansions of scattered waves in terms of QNMs were established for Schwarzschild--de Sitter space by Bony--H\"afner \cite{bony:2008}. Later, Dyatlov constructed a meromorphic continuation of the scattering resolvent for Kerr--de Sitter metrics and analysed the distribution of QNFs \cite{dyatlov:2011:cmp,dyatlov:2012}. 

All of the aforementioned works used delicate separation of variables techniques to study QNMs, hence are not stable under perturbations. In a landmark paper \cite{vasy:2013}, Vasy proved meromorphy of a family of operators whose poles define QNFs of Kerr--de Sitter metrics. This method depends only on certain microlocal properties of the geodesic flow, which are stable under perturbations. Additionally, resolvent estimates, expansions of waves in terms of QNMs, and wavefront set properties of the resolvent were also established (not to mention other applications, for instance to asymptotically hyperbolic spaces).
\begin{center}
	\begin{figure}
			 \adjustbox{trim={0cm} {.8cm} {0cm} {.75cm},clip}{\includegraphics[width=2.5in]{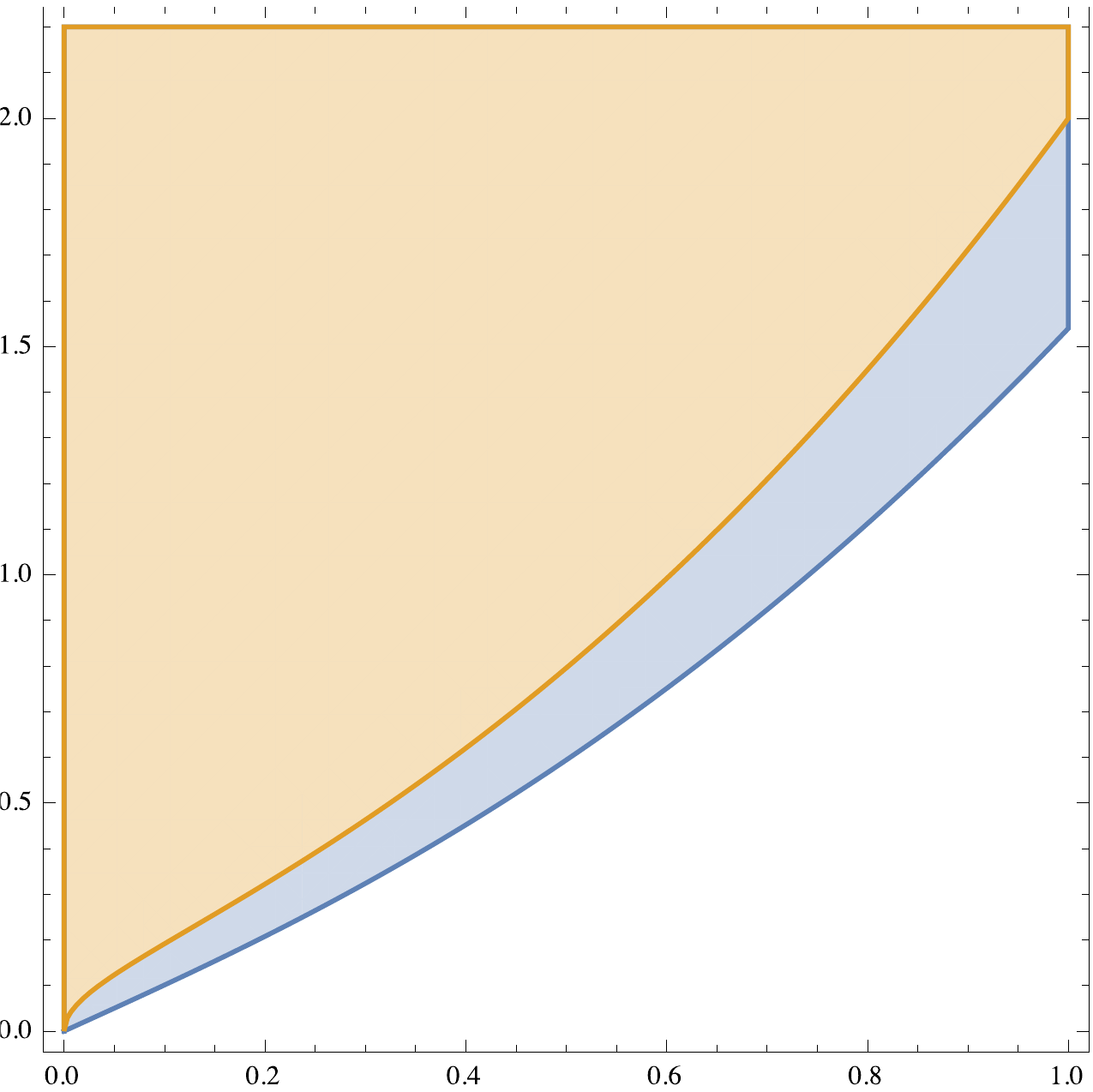}}
		\hspace{0.3in} \adjustbox{trim={0cm} {.8cm} {0cm} {.8cm},clip}{\includegraphics[width=2.5in]{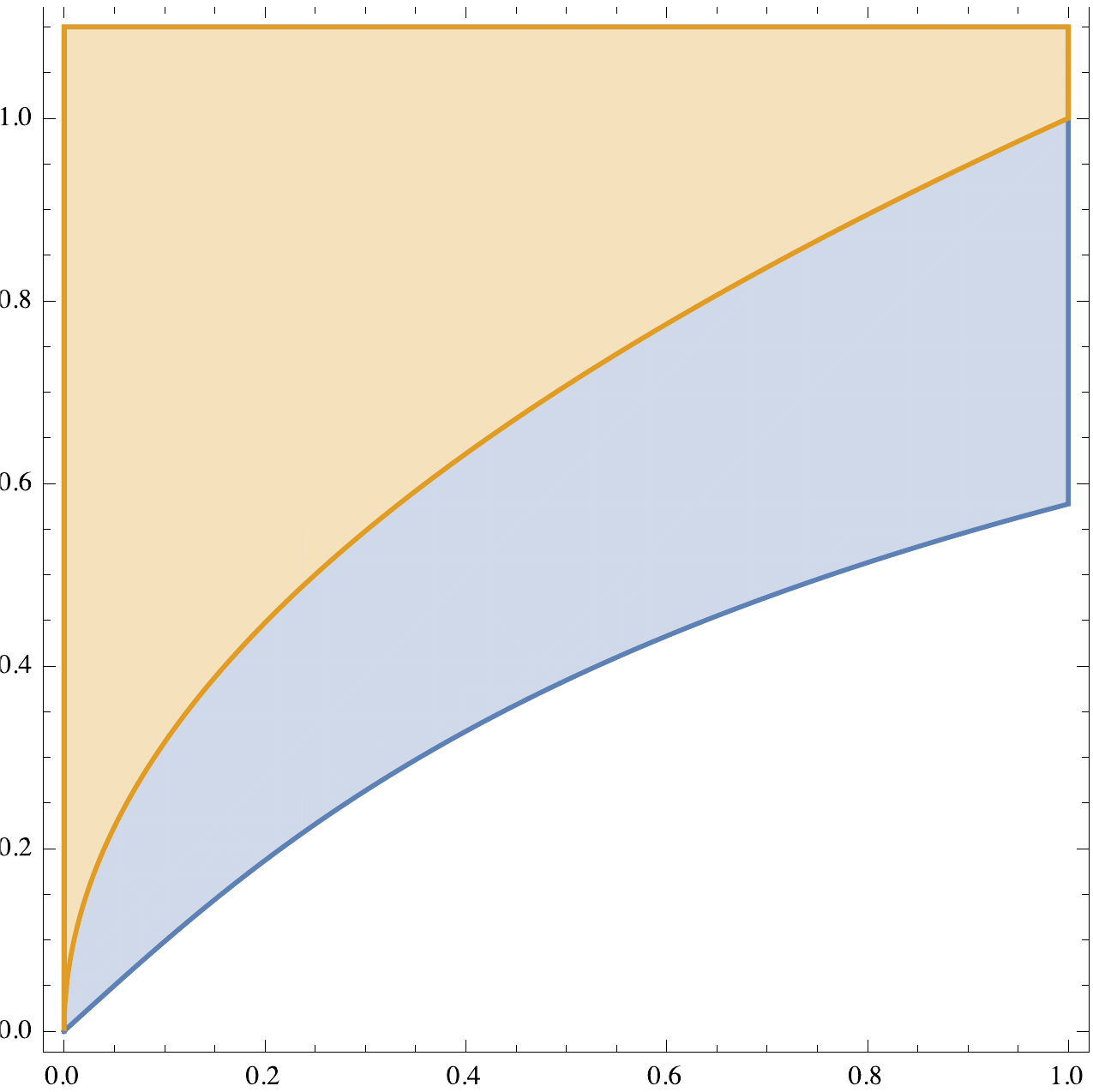}}
		\caption{Two plots showing the range of parameters $(a,l,M)$, or equivalently $(a,l,r_+)$, for which meromorphy holds. On the left is a plot of $|a|/l$ vs. $M/l$ and on the right is a plot of $|a|/l$ vs. $r_+/l$. The orange region is the regime $r_+^2 > |a| l$ for which meromorphy was established in \cite{warnick:2015:cmp}. The addition of the blue region represents the full range of admissible parameters.}
		\label{fig:test}
	\end{figure}
\end{center}

For non-rotating Schwarzschild--AdS black holes, QNMs were treated mathematically by the author in \cite{gannot:2014} using the Regge-Wheeler formalism \cite{giammatteo:2005} (separation of variables). The Regge--Wheeler equations at a fixed angular momentum $\ell$ in the nonrotating case fit into the framework of classical one-dimensional scattering theory. It was shown that the scattering resolvent exists and its restriction to a fixed space of spherical harmonics forms a meromorphic family of operators \cite[Section 4]{gannot:2014}. Therefore discreteness of QNFs for $\ell$ fixed is solved by identifying them as poles of this resolvent. Furthermore, there exist sequences of QNFs converging exponentially to the real axis, with a precise description of their real parts. In \cite{gannot:2014}, only Dirichlet boundary conditions were considered at the conformal boundary.

For general black hole backgrounds with asymptotically AdS ends, a global definition and discreteness of QNFs were studied by Warnick \cite{warnick:2015:cmp}. There, QNFs are defined as eigenvalues of an infinitesimal generator whose associated semigroup solves a mixed initial boundary value problem for the linear wave equation. When applied to the special class of Kerr--AdS metrics, there are two main results:
\begin{enumerate} \itemsep6pt
	\item QNFs at a \emph{fixed} axial Fourier mode $m \in \mathbb{Z}$ are discrete. This holds for all rotation speeds satisfying the regularity condition $|a| < l$. More generally, it holds for a more general class of ``locally stationary'' asymptotically AdS black holes, once the notion of a Fourier mode is appropriately generalized --- these spacetimes have some additional symmetries.   
	\item The set of all QNFs is discrete provided the rotation speed satisfies the Hawking--Reall bound $|a| < \min\{l,r_+^2/l\}$. These Kerr--AdS metrics admit a globally causal Killing field; this remarkable property is not shared by either the Kerr or Kerr-de Sitter family of metrics as soon as $a \neq 0$.
\end{enumerate} 
Furthermore, self-adjoint boundary conditions of Dirichlet or Robin type could be imposed at the conformal boundary. As mentioned above, this paper generalizes \cite{warnick:2015:cmp} in two ways: the QNF spectrum is shown to be discrete for rotation speeds satisfying $|a|<l$, and when $0< \nu < 1$ this discreteness holds for a broader class of boundary conditions than considered in \cite{warnick:2015:cmp}.

\section{Preliminaries}

\subsection{Microlocal preliminaries} \label{subsect:microlocal}
The purpose of this section is to fix notation for the necessary microlocal analysis. For a detailed introduction to this subject, the reader is referred to \cite[Section 18.1]{hormanderIII:1985}, \cite[Chapter 1]{shubin:2001}.

If $X$ is a smooth manifold, $\Psi^m(X)$ will denote the algebra of properly supported pseudodifferential operators of order $m$ on $X$. Denote by $\sigma_m$ the principal symbol map, fitting into the usual short exact sequence
\[
0 \rightarrow \Psi^{m-1}(X) \rightarrow \Psi^{m}(X) \xrightarrow{\sigma_m} S^m(T^*X)/S^{m-1}(T^*X) \rightarrow 0,
\]
where $S^m(T^*X)$ is the space of Kohn--Nirenberg symbols on $T^*X$. In applications, all pseudodifferential operators will be compactly supported, namely their Schwartz kernels have compact support in $X\times X$.

%
%
%
%

Let $S^*X = (T^*X \setminus 0)/\mathbb{R}_+$ denote the cosphere bundle, where $\mathbb{R}_+$ acts on $T^*X \setminus 0$ by positive dilations in the fibers. Conic subsets of $T^*X \setminus 0$ are in one-to-one correspondence with subsets of $S^*X$ via the canonical projection
\[
\kappa: T^*X \setminus 0 \rightarrow S^*X.
\]
If $a \in S^m(T^*X)$ is homogeneous of degree $m$ in the fibers, then the integral curves of the Hamilton vector field $H_a$ through $(x,\xi)$ and $(x,\mu \xi)$ with $\mu > 0$ have the same image in $S^*X$. Furthermore, the vector field $H_a$ is homogeneous of degree $m-1$, so $|\xi|^{1-m}H_a$ descends to a vector field on $S^*X$, and integral curves of $H_a$ on $T^*X \setminus 0$ are uniquely determined by those of $|\xi|^{1-m}H_a$ on $S^*X$ (up to parametrization); here $|\cdot|$ is a fixed norm on the fibers of $T^*X$.

 A symbol $a \in S^m(T^*X)$ is said to be elliptic at $(x_0,\xi_0) \in T^*X\setminus 0$ if there exists an open conic neighborhood $U\subseteq T^*X\setminus 0$ of $(x_0,\xi_0)$ such that $|\xi|^{-m}|a| \geq c > 0$ in $U$, provided $|\xi|$ is sufficiently large. This condition does not change if $a$ is modified by an element of $S^{m-1}(T^*X)$. If $a = \sigma_m(A)$, then $\ELL(A)\subseteq T^*X\setminus 0$ will denote the set of elliptic points of $a$. The characteristic set $\Sigma(A)$ is the complement in $T^*X\setminus 0$ of $\ELL(A)$, which is thus a closed conic subset of $T^*X\setminus 0$. If $a$ is homogeneous of degree $m$, then $(x_0,\xi_0) \in \Sigma(A)$ if and only if $a(x_0,\xi_0)= 0$.
 
 Given $a \in S^m(T^*X)$, say that $(x_0,\xi_0) \in T^*X\setminus 0$ is not in the essential support of $a$ if there exists an open conic neighborhood $U \subseteq T^*X \setminus 0$ of $(x_0,\xi_0)$ such that
 \[
 |a(x,\xi)| \leq C_N \left<\xi \right>^{-N}
 \] 
 for each $N$, uniformly near $U$. The wavefront set $\WF(A)$ of $A \in \Psi^m(X)$ is defined as the essential support of its full symbol in any local coordinate chart. Thus $A$ is negligible outside of $\WF(A)$ in a precise microlocal sense.
 
The simplest microlocal estimate controls $u$ in some region of phase space in terms $Pu$, provided $P$ is elliptic in a neighborhood of that region. More precisely, one has the following standard elliptic estimate:
 \begin{prop} [{\cite[Theorem 18.1.24']{hormanderIII:1985}}] \label{prop:microelliptic} Suppose that $P \in \Psi^m(X)$ is properly supported, $A,\, G \in \Psi^0(X)$ are compactly supported, and
 \[
 \WF(A) \subseteq \ELL(P) \cap \ELL(G).
 \]	
 If $u \in \mathcal{D}'(X)$ satisfies $GPu \in H^{s-m}(X)$ for some $s$, then $Au \in H^s(X)$. Moreover, there exists $\chi \in C_c^\infty(X)$ such that
 \begin{equation} \label{eq:ellipticestimate}
 \| Au \|_{H^s(X)} \leq C \left(\|GPu \|_{H^{s-m}(X)} + \| \chi u \|_{H^{-N}(X)} \right)
 \end{equation}
 for each $N$.
 \end{prop}
 
 Observe that each of the terms in \eqref{eq:ellipticestimate} has support in a fixed compact subset of $X$, hence there is no ambiguity in the Sobolev norms.

 Next is the Duistermaat--H\"ormander theorem on propagation of singularities.
 
 \begin{prop} [{\cite[Theorem 26.1.4]{hormander2009analysis}}] \label{prop:microprop}
 Suppose that $P \in \Psi^m(X)$ is properly supported and $A,\,B,\,G \in \Psi^0(X)$ are compactly supported. Assume that $\sigma_m(P)$ has a real-valued homogeneous representative $p$, and that for each $(x,\xi) \in \WF(A)$ there exists $T\geq 0$ such that 
 \begin{itemize} \itemsep6pt
 	\item $\exp(T  H_p)(x,\xi) \in \ELL(B)$,
 	\item $\exp(t H_p)(x,\xi) \in \ELL(G)$ for each $t \in [0,T]$.
 \end{itemize}
 If $u \in \mathcal{D}'(X)$ satisfies $GPu \in H^{s-m+1}(X)$ and $Bu \in H^s(X)$, then $Au \in H^s(X)$. Moreover, there exists $\chi \in C_c^\infty(X)$ such that
 \[
 \| Au \|_{H^s(X)} \leq C \left( \|GPu \|_{H^{s-m+1}(X)} + \| Bu \|_{H^s(X)}  + \| \chi u \|_{H^{-N}(X)} \right)
 \]
 for each $N$.
\end{prop}

\subsection{Parameter-dependent differential operators} \label{subsubsect:param} Recall the class of parameter-dependent differential operators on $X$: these are operators $P(\lambda)$ given in local coordinates by
\[
P(x,D_x,\lambda) = \sum_{j + |\alpha| \leq m} a_{j,\alpha}(x) \lambda^j D_x^\alpha,
\]
where $\lambda \in \mathbb{C}$ is a parameter; the order of $P(\lambda)$ is said to be at most $m$, and the set of all such operators is denoted $\mathrm{Diff}^m_{(\lambda)}(X)$. The parameter-dependent principal symbol of $P(\lambda)$ is given in coordinates by
\[
\sigma_{m}^{(\lambda)}(P(\lambda)) = \sum_{j+|\alpha| = m} a_{j,\alpha}(x) \lambda^j \xi^\alpha.
\]
This is a well-defined function on $T^*X \times \mathbb{C}_\lambda$, which is a homogeneous degree $m$ polynomial in the fibers. If $P(\lambda) \in \mathrm{Diff}^m_{(\lambda)}(X)$ has parameter-dependent principal symbol $p(\lambda) = p(x,\xi;\lambda)$ and $\Lambda \subseteq \mathbb{C}$ is an angular sector, then $P(\lambda)$ is said to be parameter-elliptic on an open subset $U \subseteq X$ with respect to $\Lambda$ if 
\[
p(x,\xi;\lambda) \neq 0,\quad (x,\xi,\lambda) \in (T^*_U X \times \Lambda) \setminus 0.
\]
Of course parameter-ellipticity with respect to any $\Lambda$ also implies ellipticity in the sense of Section \ref{subsect:microlocal}.


For the corresponding class of parameter-dependent pseudodifferential operators, see \cite[Section 9]{shubin:2001}. The closely related semiclassical calculus is treated in \cite[Chapter 6]{dimassi:1999}, \cite{zworski:2012}, and \cite[Appendix E]{zworski:resonances} for example.


%
%

\subsection{Lorentzian metrics} \label{subsubsect:lorenztian} 

Let $g$ denote a Lorentzian metric of signature $(1,n)$ on an $n+1$ dimensional manifold $\mathcal{M}$ with a complete Killing field $T$. Assume there exists a spacelike hypersurface $X \subseteq \mathcal{M}$ such that each integral curve of $T$ intersects $X$ exactly once. Then the parameter along the flow of $T$ defines a function $t: \mathcal{M} \rightarrow \mathbb{R}$ such that $X = \{ t= 0\}$. Moreover, the flow gives a diffeomorphism $\mathcal{M} = \mathbb{R}_t \times X$. In this product decomposition, $T = \partial_t$.

With respect to the splitting $T^*\mathcal{M} = \mathbb{R}\cdot dt \oplus T^*X$, the principal symbol of the wave operator $\Box_g$ (which does not depend on $t$) is given by
\[
\sigma_2(\Box_g)(x,\xi,\tau) = -g^{-1}(\xi \cdot dx + \tau \, dt, \xi \cdot dx + \tau \, dt),
\]
where $\tau\in \mathbb{R}$ is the momentum conjugate to $t$. Let $\widehat{\Box}_g(\lambda)$ denote the operator obtained from $\Box_g$ by replacing $D_t$ with $-\lambda$. Thus $\widehat{\Box}_g(\lambda)$ acts on $u \in C^\infty(X)$ by 
\[
\widehat{\Box}_g(\lambda) u = e^{i\lambda t}\Box_g e^{-i\lambda t }u,
\]
where $u$ is identified with a $T$-invariant function on $\mathcal{M}$.
This is a parameter-dependent differential operator of order two in the sense of Section \ref{subsubsect:param}, whose parameter-dependent principal symbol is just $p(x,\xi;\lambda) = \sigma_2(\Box_g)(x,\xi,-\lambda)$. In particular,
\begin{align} \label{eq:realimaginaryparts}
\Re p(x,\xi;\lambda) &= -g^{-1}(\xi \cdot dx -\Re \lambda\, dt, \xi \cdot dx - \Re \lambda \,dt) + (\Im \lambda)^2 g^{-1}(dt,dt), \notag \\
\Im p(x,\xi;\lambda) &= 2\,(\Im \lambda) g^{-1}(\xi \cdot dx - \Re \lambda \, dt, dt). 
\end{align}
The standard principal symbol of $P(\lambda)$ is $\sigma_2(P(\lambda))(x,\xi) = p(x,\xi;0)$, which in particular is real-valued and independent of $\lambda$.

\begin{lem} \label{lem:lorentzianmetrics} The operator $\widehat{\Box}_{g}(\lambda)$ has the following properties.
\begin{enumerate} \itemsep6pt
	
	\item If $\Im \lambda \neq 0$, then $p(x,\xi;\lambda) \neq 0$ for $\xi \in T^*_x X$.
	
	\item If $T$ is timelike at $x\in X$, then $p(x,\xi;0) \neq 0$ for $\xi \in T^*_x X \setminus 0$.
\end{enumerate}
\end{lem}
\begin{proof}
	\begin{inparaenum}[(1)]
		\item Recall that $g^{-1}(dt, dt) > 0$ since $X$ is spacelike. If $\Im p(\lambda) = 0$ and $\Im \lambda \neq 0$, then $\xi \cdot dx - \Re \lambda \, dt$ would be orthogonal to the timelike vector $dt$, hence spacelike. This means 
		\[
		g^{-1}(\xi \cdot dx - \Re \lambda \, dt,\xi \cdot dx - \Re \lambda\, dt) < 0,
		\] 
		which shows that $\Re p(\lambda) \neq 0$.
		
		\item Note that $\lambda= g^{-1}(T^{\,\flat},\xi \cdot dx - \lambda\,dt)$, where $T^{\,\flat}$ is the covector obtained from $T$ by lowering an index. If $\lambda = 0$ and $T$ is timelike, then $\xi \cdot dx$ is spacelike, so $p(x,\xi;0)$ is positive definite. 
	\end{inparaenum}
\end{proof}

	As a corollary of Lemma \ref{lem:lorentzianmetrics}, if $T$ is timelike at $x$ and $\Lambda$ is an angular sector disjoint from $\mathbb{R} \setminus 0$, then $\widehat{\Box}_g(\lambda)$ is parameter-elliptic near $x$ with respect to $\Lambda$.

\section{Local theory of Bessel operators}
  \label{subsect:bessel} This section reviews some facts about differential operators with inverse square singularities. General elliptic boundary value problems for this class of Bessel operators were recently studied in \cite{gannot:bessel}. Here only the local theory is reviewed, namely on coordinate patches. This is meant to acquaint the reader with the basic objects. In applications, the results of this section must be globalized via partition of unity arguments. This is briefly indicated in Section \ref{subsect:adsfunctionspaces}; for more details see \cite{gannot:bessel}.

\subsection{Basic definitions} \label{subsubsect:besseldefs}

Let $\RNP = \RN \times \RP$. A typical element $x \in \RNP$ is written $x = (x', x_n)$, where $x' = (x_1,\ldots,x_{n-1}) \in \RN$ and $x_n \in \RP$. The space $L^2(\RNP)$ of square integrable functions is defined with respect to Lebesgue measure.

For each $\nu \in \mathbb{R}$ the differential operators
\[
D_\nu = x_n^{1/2-\nu}D_{x_n} x_n^{\nu-1/2}, \quad D_\nu^* = x_{n}^{\nu-1/2} D_{x_n} x_n^{1/2-\nu}
\]
are well defined on $\RNP$. Note that  $D_\nu^*$ is indeed the formal $L^2(\RNP)$ adjoint of $D_\nu$. Formally define
\[
|D_\nu|^2 = D_{x_n}^2 + (\nu^2 - 1/4)x_n^{-2},
\]
which satisfies $|D_\nu|^2 = D_\nu^* D_\nu$. 

Now assume that $\nu > 0$, and consider a parameter-dependent operator $P(\lambda)$ on $\RNP$ of the form
\begin{equation} \label{eq:besseloperator}
P(x,D_\nu,D_{x'};\lambda) = |D_\nu|^2 + B(x,D_{x'};\lambda) D_\nu + A(x,D_{x'};\lambda),
\end{equation}
where $A(\lambda), B(\lambda)$ are parameter-dependent operators on $\OL{\RNP}$ of order two, one respectively, such that the coefficients of $B(\lambda)$ vanishes at $x_n = 0$. Such an operator will be referred to as a \emph{parameter-dependent Bessel operator of order $\nu$}. It is easy to check that the formal adjoint $P(\lambda)^*$ satisfies the same conditions as $P(\lambda)$.

\subsection{Ellipticity}

If $A(\lambda)$ is defined as in \eqref{eq:besseloperator}, let $A(\lambda)^\circ$ denote its principal part:
\[
A(x,D_{x'};\lambda)^\circ = \sum_{j+|\alpha| = 2} a_{\alpha,j}(x) \lambda^j D_{x'}^\alpha.
\]
Thus $A(x',0,\eta;\lambda)^\circ$ is a polynomial of degree two in $(\eta,\lambda) \in T_{x'}^* \RN \times \mathbb{C}$. Associated with $P(\lambda)$ is the polynomial function
\begin{equation} \label{eq:besselpolynomial}
\zeta^2 + A(x',0,\eta;\lambda)^\circ,
\end{equation}
indexed by points $x' \in \RN$. If $\Lambda \subseteq \mathbb{C}$ is an angular sector, then $P(\lambda)$ is said to be parameter-elliptic with respect to $\Lambda$ at a boundary point $x' \in \RN$ if \eqref{eq:besselpolynomial} does not vanish for $(\zeta,\eta,\lambda) \in ( \mathbb{R} \times T^*_{x'} \RN \times \Lambda ) \setminus 0$. Ellipticity at the boundary (in the standard, non-parameter-dependent sense) is defined by the condition that \eqref{eq:besselpolynomial} evaluated at $\lambda = 0$ does not vanish for $(\zeta,\eta) \in  ( \mathbb{R} \times T^*_{x'} \RN ) \setminus 0$.

If $P(\lambda)$ is parameter-elliptic at the boundary, then \eqref{eq:besselpolynomial} (as a function of $\zeta$) has two non-real roots $\pm \zeta(x',\eta;\lambda)$ for $(\eta,\lambda) \in (T^*_{x'} \RN \times \Lambda) \setminus 0$. By convention $\Im \zeta(x',\eta;\lambda) < 0$. Any solution to the ordinary differential equation
\begin{equation} \label{eq:modelODE}
\left( |D_\nu|^2 + A(x',0,\eta;\lambda)^\circ \right) u = 0
\end{equation}
is a linear combination of Bessel functions
\[
u = c_+ x_n^{1/2}I_\nu(i\zeta(x',\eta;\lambda)x_n) + c_- x_n^{1/2}K_\nu(i\zeta(x',\eta;\lambda)x_n).
\]
Requiring that $u$ is square integrable on $\mathbb{R}_+$ near infinity with respect to ordinary Lebesgue measure implies that $c_+ = 0$; this follows from the asymptotics of Bessel functions \cite[Chapter 7.8]{olver:2014}. Furthermore, if $\nu \geq 1$, then square integrability near $x_n = 0$ implies also $c_- = 0$. If $0 < \nu < 1$, then the space of solutions to \eqref{eq:modelODE} is one dimensional, and boundary conditions must be imposed along $\RN$.

\subsection{Boundary operators} \label{subsubsect:boundaryoperators}

When $0 < \nu < 1$, one needs to impose boundary conditions to prove coercive estimates. Formally define the weighted restrictions 
\[
\gamma_- u = x_n^{\nu-1/2} u|_{\mathbb{R}^{n-1}}, \quad \gamma_+ u = -x_n^{1-2\nu} \partial_{x_n}(x_n^{\nu-1/2} u)|_{\mathbb{R}^{n-1}}.
\]
The boundary operator $T(x',D_{x'};\lambda)$ is written as
\[
T(\lambda) = T_1 + \lambda T_0
\]
for $T_0,T_1$ of the following forms:
\[
T_1 = T_1^+ \gamma_+ + T_1^- \gamma_-, \quad T_0 = T_0^- \gamma_-,
\]
where $T_1^+, T_0^-$ are smooth functions on $\RN$ and $T_1^-$ is a first order operator on $\RN$. 

Depending on the value of $\nu$, different terms should be considered as the ``principal part'' of $T(\lambda)$. Fix the smallest $\mu \in \{1- \nu ,2-\nu,1+\nu\}$ such that the orders of $T_1^- + \lambda T_0^-$ and $T_1^+$ do not exceed $\mu- 1 + \nu$ and $\mu-1-\nu$, respectively. Here order is taken in the sense of parameter-dependent differential operators on the boundary. Given $\mu$ as above, define $T(\lambda)^\circ = T(x',D_{x'};\lambda)^\circ$ to be the boundary operator which for each $(x',\eta) \in T^*\RN$ satisfies
\[
T(x',\eta;\lambda)^\circ = \sigma^{(\lambda)}_{\lceil \mu - 1 + \nu\rceil}(T_1^- + \lambda T^-_0) \gamma_- +  \sigma^{(\lambda)}_{\lceil \mu - 1 - \nu\rceil}(T_1^+)\gamma_+.
\]
This is the principal part of $T(\lambda)$ --- see \cite[Section 4]{gannot:bessel} for more details.

\subsection{Lopatinski\v{\i} condition} \label{subsubsect:lopatinskii}
 Let $0 < \nu < 1$ and suppose that $T(\lambda)$ is a boundary operator with principal part $T(\lambda)^\circ$ as in Section \ref{subsubsect:boundaryoperators}. If $P(\lambda)$ is parameter-elliptic at the boundary with respect to $\Lambda$, then $T(\lambda)$ is said to satisfy the Lopatinski{\v{\i}} condition with respect to $P(\lambda)$ if for each $(x',\eta,\lambda) \in  \left( T^* \RN \times \Lambda \right) \setminus 0$ the only solution to the problem
\[
\begin{cases}
\left( |D_\nu|^2 + A(x',0,\eta;\lambda)^\circ \right) u = 0,\\
T(x',\eta; \lambda )^\circ u = 0,\\
u(x_n) \text{ is bounded as $x_n \rightarrow \infty$}
\end{cases}
\]
is the trivial solution $u = 0$. In that case, the operator 
\[
\mathscr{P}(\lambda) = \begin{pmatrix}
P(\lambda) \\ T(\lambda)
\end{pmatrix}
\]
is said to be parameter-elliptic at the boundary with respect to $\Lambda$.  Similarly, the Lopatinski{\v \i} condition and ellipticity at $Y$ (in the standard, non-parameter-dependent sense) are defined by taking $\lambda =0$ above. The basic consequences of ellipticity in this sense (on a compact manifold with boundary) are proved in \cite[Section 4]{gannot:bessel}, and exploited in Section \ref{subsect:adsellipticestimate} of this paper.

\section{Kerr--AdS spacetime} \label{sect:kerradsmetric}

The Kerr--AdS metric is determined by three parameters: \begin{inparaenum}[(i)] \item $\Lambda < 0$, the negative cosmological constant, \item  $M > 0$, the black hole mass, \item $a \in \mathbb{R}$, the angular momentum per unit mass. \end{inparaenum} Given parameters $(\Lambda,M, a)$, let $l^2 = 3/|\Lambda|$ and introduce the quantities 
\begin{gather*}
\Delta_r = (r^2+a^2)\left(1+ \frac{r^2}{l^2}\right) - 2M r;  \quad 
\Delta_\theta = 1- \frac{a^2}{l^2}\cos^2\theta; \\
\varrho^2 = r^2 + a^2 \cos^2 \theta; \quad \alpha = \frac{a^2}{l^2}.
\end{gather*}
The following observation concerns the location of roots of $\Delta_r$.

\begin{lem} \label{lem:roots}
	Any real root of $\Delta_r$ must be nonnegative, and there at most two real roots. If $a = 0$, then $\Delta_r$ always has a unique positive root.
\end{lem}
\begin{proof}
	\begin{inparaenum}[(1)] \item When $a=0$ it is clear that $\Delta_r$ has a unique positive root, and furthermore $\Delta_r'(r) > 0$ for $r > 0$. 
	
	\item On the other hand, if $a\neq 0$ then $\Delta_r(0) > 0$ and $\Delta_r'(0) < 0$. At the same time, $\Delta_r(r) \rightarrow \infty$. Since $\Delta_r'' > 0$, when $a\neq 0$ any real root of $\Delta_r$ must be positive, and there are at most two real roots. \end{inparaenum}
\end{proof}

\noindent Let $r_+$ denote the largest positive root of $\Delta_r$, when it exists. Throughout, it is assumed that\begin{inparaenum}[(i)]\item $r_+$ exists and $ \Delta_r'(r_+) > 0$, \item the rotation speed satisfies the regularity condition $|a| < l$ \end{inparaenum}.

 The Kerr--AdS metric determined by $(\Lambda,M,a)$ is initially defined on 
 \[
\mathcal{M}_0 = \mathbb{R} \times (r_+,\infty) \times \mathbb{S}^2.
\] 
Let $t$ and $r$ denote standard coordinates on $\mathbb{R}$ and $(r_+,\infty)$ respectively. Away from the north and south poles of $\mathbb{S}^2$, let $(\theta, {\phi})$ denote usual spherical coordinates. Thus $\theta \in (0,\pi)$ and $\phi \in \mathbb{R}/(2\pi \mathbb{Z})$, where these coordinates degenerate as $\theta$ tends to either $0$ or $\pi$. In terms of Boyer--Lindquist coordinates $({t},r,\theta,{\phi})$, the metric $g$ is given by
\begin{align*}
g  =& -\varrho^2 \left( \frac{d r^2}{\Delta_r} + \frac{d\theta^2}{\Delta_\theta} \right) - \frac{\Delta_\theta \sin^2 \theta}{\varrho^2(1-\alpha)^2 }\left(a\, dt - (r^2 + a^2) \, d{ \phi}\right)^2 \\&+ \frac{\Delta_r}{\varrho^2(1-\alpha)^2 }\left(d t - a\sin^2\theta \, d{\phi}\right)^2.
\end{align*}
Introducing Cartesian coordinates near the north and south poles of $\mathbb{S}^2$ shows that $g$ extends smoothly to those coordinate singularities. The dual metric $g^{-1}$ is given by
\begin{align} \label{eq:dualmetric}
g^{-1} = & \ -\frac{\Delta_r}{\varrho^2} \partial_{r}^2 - \frac{\Delta_{\theta}}{\varrho^2} \partial_{\theta}^2 - \frac{(1-\alpha)^2}{\varrho^2 \Delta_\theta \sin^2  \theta} \left( a \sin^2 \theta \partial_{{t}} + \partial_{{\phi}}\right)^2 \notag\\&+ \frac{(1-\alpha)^2}{\varrho^2\Delta_r}\left((r^2+a^2)\partial_{{t}} + a \partial_{{\phi}}\right)^2.
\end{align}
Also observe that the scaling transformations 
	\[
	l \mapsto \sigma l,\quad a \mapsto \sigma a,\quad  M \mapsto \sigma M,\quad r \mapsto \sigma r,\quad t \mapsto \sigma t
	\] 
	induce a conformal transformation $g \mapsto \sigma^2 g$. By setting $\sigma = l^{-1}$, it is assumed for the remainder of the paper that $l = 1$, or equivalently $|\Lambda| = 3$.

\subsection{Kerr--AdS as an asymptotically anti-de Sitter spacetime} \label{subsect:kerrasaads}

To analyze the behavior of $g$ for large $r$, introduce a new radial coordinate $s=r^{-1}$. Let
\[
\mathcal{I} = \{s = 0 \},
\]
which may be glued to $\mathcal{M}_0$ as a boundary component. In particular, $s$ is a global boundary defining function for $\scri$. Noting that 
\begin{equation} \label{eq:evenasymptotics}
\varrho^2 = s^{-2} + \mathcal{O}(1), \quad \Delta_r = s^{-4} + \mathcal{O}(s^{-2}),
\end{equation}
it follows that $s^2 g$ has a smooth extension to $\mathcal{M}_0 \cup \mathcal{I}$, which justifies calling $\scri$ a conformal boundary for $\mathcal{M}_0$. Moreover, $s^{-2}g^{-1}(ds,ds) \rightarrow -1$ as $s \rightarrow 0$, which shows that the restriction of $s^2 g$ to $T \scri$ is a Lorentzian metric on $\scri$. Thus $g$ has an asymptotically anti-de Sitter end at $r \rightarrow \infty$ in the sense of \cite{friedrich1995einstein}.

\subsection{Extension across the event horizon} \label{subsect:extension} As usual, $g$ appears singular at the event horizon 
\[
\mathcal{H}^+ = \{ r = r_+ \} = \{ \Delta_r = 0\}.
\]
The metric may be extended smoothly across this hypersurface by making an appropriate change of variables. Set
\begin{equation} \label{eq:kerrstar}
t^\star = t + F_t(r); \quad \phi^\star =  \phi + F_\phi(r),
\end{equation}
where $F_t,\, F_\phi$ are smooth functions on $(r_+,\infty)$ satisfying the following conditions:
\begin{enumerate} \itemsep6pt
	\item For some smooth function $f_+(r)$,
	\begin{equation} \label{eq:changeofcoords}
F_t'(r) = \frac{1-\alpha}{\Delta_r}(r^2+a^2) + f_+(r), \quad  F_\phi'(r) = a\frac{1-\alpha}{\Delta_r}
\end{equation}
near $r_+$,
\item $F_t(r) = F_\phi(r) = 0$ for $r$ sufficiently large.
\end{enumerate}
In the region where \eqref{eq:changeofcoords} is valid, the dual metric in $(t^\star,r,\theta,\phi^\star)$ coordinates reads
\begin{align} \label{eq:extendedmetric}
\varrho^2 g^{-1} = &-\Delta_r 
\left( \partial_r + f_+ \partial_{t^\star} \right)^2 - \Delta_\theta \partial_\theta^2 - 2(1-\alpha) \left( \partial_r + f_+ \partial_{t^\star} \right)\left((r^2+a^2)\partial_{t^\star}  +a \partial_{\phi^\star} \right) \notag \\ &- \frac{(1-\alpha)^2}{ \Delta_\theta \sin^2 \theta} \left( a \sin^2 \theta \partial_{t^\star} +  \partial_{{\phi^\star}}\right)^2.
\end{align}
This expression is smooth up to $\mathcal{H}^+$. In fact, given $\delta >0$ sufficiently small, let
\begin{equation} \label{eq:foliation}
\mathcal{M}_\delta = \mathbb{R} \times (r_+ -\delta,\infty)\times \mathbb{S}^2.
\end{equation}
If $t^\star$ is the coordinate on $\mathbb{R}$, then \eqref{eq:extendedmetric} defines a dual Lorentzian metric on $\mathcal{M}_\delta$. Geometrically, $\mathcal{M}_\delta$ is foliated by translations of $\{ t^\star = 0\}$ along integral curves of $\partial_{t^\star}$, which gives the product decomposition \eqref{eq:foliation}. Choices of $F_t$ correspond to foliations of $\mathcal{M}_\delta$ by different initial hypersurfaces.

Let $X_\delta = \{t^\star = 0\} \subset \mathcal{M}_\delta$. For the purposes of this paper, $t^\star$ must be chosen so that $X_\delta$ is spacelike, or equivalently $g^{-1}(dt^\star,dt^\star) > 0$. To accomplish this, choose $F_t$ such that $F_t'$ satisfies \eqref{eq:changeofcoords} globally, where 
\begin{equation} \label{eq:f_+larger}
f_+(r) = \frac{\alpha-1}{\Delta_r}(r^2+a^2)
\end{equation}
for $r$ sufficiently large, and
\begin{equation}  \label{eq:dttimelike}
\Delta_r f_+^2 + 2(1-\alpha)(r^2+a^2)f_+ < - (1-\alpha)^2 a^2.
\end{equation}
Since $|a|< 1$, any function $f_+(r) \sim (\alpha-1)r^{-2}$ satisfies \eqref{eq:dttimelike} for $r$ sufficiently large. Interpolating between $f_+(r) = (\alpha-1)(1+r^2)^{-1}$ near $r_+$ and \eqref{eq:f_+larger} for large $r$ finishes the construction.

\subsection{Surface gravity} The hypersurface $\mathcal{H}^+$ is a Killing horizon generated by the future-pointing Killing vector field
\begin{equation} \label{eq:generator}
K = \partial_{t^\star} + \frac{a}{r_+^2 + a^2}\partial_{\phi^\star}.
\end{equation}
This means that $\mathcal{H}^+$ is a $K$-invariant null hypersurface and $K$ is normal to $\mathcal{H}^+$. These conditions imply that
\begin{equation} \label{eq:surfacegravityeqn}
\grad_g g(K,K) = -2\varkappa K
\end{equation}
on $\mathcal{H}^+$ for some function $\varkappa$. Examining the $\partial_{t^\star}$ component of \eqref{eq:surfacegravityeqn} on the horizon gives the (constant) value
\begin{equation} \label{eq:surfacegravity}
\varkappa = \frac{\Delta_r'(r_+)}{2(1-\alpha)(r_+^2+a^2)},
\end{equation} 
which is positive under the assumption that $\Delta_r'(r_+) > 0$.

\subsection{The manifold with boundary} \label{subsect:manifoldwithboundary}
As indicated in Section \ref{subsect:mainresults}, $\mathcal{M}_\delta$ is profitably viewed as the interior of the manifold
\[
\OL{\mathcal{M}}_\delta = \mathcal{M}_\delta \cup \mathcal{H}_\delta \cup \scri,
\]
where $\mathcal{H}_\delta = \{ r = r_+ - \delta\}$ and $\scri = \{ s = 0\}$. The metric $g$ is smooth up to $\mathcal{H}_\delta$, and $s^2 g $ is smooth up to $\scri$. Observe that $\mathcal{H}^+ = \mathcal{H}_0$, and if $\delta > 0$, then $dr$ is timelike in the region bounded by $\mathcal{H}^+$ and $\mathcal{H}_\delta$.

In terms of the time slicing, $t^\star$ extends to a function on $\OL{\mathcal{M}}_\delta$, and the level set $\OL{X}_\delta = \{ t^\star = 0 \} \subseteq \OL{\mathcal{M}}_\delta$ is compact and spacelike with respect to $s^2 g$. The interior of $\OL{X}_\delta$ is identified with $X_\delta$, and $\partial X_\delta = H_\delta \cup Y$, where
$H_\delta = \mathcal{H}_\delta \cap \OL{X}_\delta$ and $Y = \scri \cap \OL{X}_\delta$.

\subsection{Klein--Gordon equation} \label{subsect:kleingordon}

The main object of study is the Klein--Gordon equation
\begin{equation} \label{eq:kleingordon}
\left( \Box_g + \nu^2 - 9/4 \right) \phi = 0.
\end{equation}
The mass term is written as $\nu^2 - 9/4$ to emphasize the importance of the parameter $\nu$, which is required to be strictly positive. By choosing an extension $\mathcal{M}_\delta$ of $\mathcal{M}_0$ according to Section \ref{subsect:extension}, the Klein--Gordon equation \eqref{eq:kleingordon} continues to make sense on $\mathcal{M}_\delta$.

Since this paper is ultimately concerned with quasinormal modes (which solve the homogeneous equation \eqref{eq:kleingordon}), it is more convenient to work with the operator $P(\lambda)$ given by
\[
P(\lambda) = \varrho^2 \big(\widehat\Box_g(\lambda) +\nu^2 - 9/4\big),  
\]
where $\widehat{\Box}_g(\lambda)$ is defined in Section \ref{subsubsect:lorenztian}. Up to a multiplicative factor, this is the spectral family of the Klein--Gordon equation \eqref{eq:kleingordon} acting on $X_\delta$. Multiplication by a positive prefactor growing like $r^2$ ensures that $P(\lambda)$ will be a Fredholm operator between $L^2$ based spaces with the same $r$-weights. The particular choice $\varrho^2 \sim r^2$ simplifies some formulae.

If $dS_t$ is the measure induced on $X_\delta$ by the metric, let $\mathcal{L}^2(X_\delta)$ denote square integrable functions with respect to $\varrho^{-2} A \cdot dS_t$, where
\[
 A = g^{-1}(dt^\star,dt^\star)^{-1/2}.
\]
Then the formal adjoint of $P(\lambda)$ satisfies 
\[
P(\lambda)^* = P(\bar \lambda).
\]
This follows from the relationship $|\det g| = A^2 \, |\det h|$, where $h$ is the induced metric on $X_\delta$, and the self-adjointness of $\Box_g$ with respect to the volume form on $\mathcal{M}_\delta$. Observe that $\Sob^0(X_\delta) = \mathcal{L}^2(X_\delta)$ is equivalent as a Hilbert space to the one defined in Section \ref{subsect:mainresults}, see \eqref{eq:L2space} in particular. It is precisely this space for which finite energy solutions to \eqref{eq:kleingordon} are square integrable.

\section{Microlocal study of $P(\lambda)$} \label{sect:microlocalstudy}

The purpose of this section is to understand the microlocal structure of $P(\lambda)$. Unless otherwise stated, all the analysis take place on the extended time slice $X_\delta$ with $\delta > 0$ fixed (the only exceptions are Lemmas \ref{lem:qnm_def:invertible1}, \ref{lem:qnm_def:invertible2}, where $\delta = 0$ is allowed). Let $p = \sigma_2(P(\lambda))$ denote the homogeneous principal symbol of $P(\lambda)$, which observe is real-valued and independent of $\lambda$. Explicitly,
\begin{equation} \label{eq:nonsemiclassicalsymbol}
p(x,\xi) = \Delta_r {\xi}_r^2 + 2 a (1-\alpha) {\xi}_r {\xi}_{\phi^\star} + \Delta_\theta {\xi_\theta}^2 + \frac{(1-\alpha)^2}{\Delta_\theta \sin^2\theta} {\xi}_{\phi^\star}^2,
\end{equation}
where $(\xi_r,\xi_\theta,\xi_{\phi^\star})$ are momenta dual to $(r,\theta,\phi^\star)$.

\subsection{Characteristic set}

Let $\Sigma = \{ p = 0\} \setminus 0$ denote the characteristic set of $P(\lambda)$. Its image in $S^*X_\delta$ is denoted by
\[
\widehat{\Sigma} = \kappa(\{ p = 0\} \setminus 0) \subseteq S^*X_\delta.
\]
Observe that $\xi_r \neq 0$ on $\Sigma$, since from \eqref{eq:nonsemiclassicalsymbol} the conditions $p = 0$ and $\xi_r = 0$ force $\xi =0$. Therefore $\Sigma$ is the disjoint union 
\[
\Sigma = \Sigma_+ \cup \Sigma_-, \quad \Sigma_\pm = \Sigma \cap \{ \pm \xi_r > 0\}.
\]
Similarly, $\widehat{\Sigma} = \widehat{\Sigma}_+ \cup \widehat{\Sigma}_-$, where $\widehat{\Sigma}_\pm =\kappa(\Sigma_\pm)$. Furthermore $\widehat{\Sigma}$ does not intersect the region where $\partial_{t^\star}$ is timelike by Lemma \ref{lem:lorentzianmetrics}. For $r> r_+$, this condition can be checked in Boyer--Lindquist coordinates, observing that $\partial_{t^\star} = \partial_t$ and the map $(t,r,\theta,\phi) \mapsto (t^\star,r,\theta,\phi^\star)$ does not affect the $r$ variable: the vector field $\partial_t$ is timelike provided
 \[
\Delta_r > a^2 \Delta_\theta \sin^2 \theta.
\]  
In particular, $\widehat{\Sigma}  \subseteq \{ \Delta_r \leq a^2 \}$.

\subsection{Null-bicharacteristic flow} \label{subsect:sourcesink}
The analysis in this section closely follows \cite[Section 6.3]{vasy:2013}, which applies to the Kerr-de Sitter family of metrics. Let 
\[
N^*(\{ r =r_+\}) \setminus 0 \subseteq T^*X_\delta \setminus 0
\]
denote the conormal bundle to $\{ r= r_+\} \subseteq X_\delta$, less the zero section. Since $\xi_r \neq 0$ on $N^*(\{ r =r_+\}) \setminus 0$, there is a splitting 
\[
N^*(\{ r =r_+\}) \setminus 0 = \mathcal{R}_+ \cup \mathcal{R}_-,
\] 
where
\[
\mathcal{R}_{\pm} = \{ r = r_+, \ \xi_\theta = \xi_{\phi^\star} = 0, \ \pm \xi_r >  0\} \subset T^*X_\delta \setminus 0.
\]
Let $L_\pm$ denote the image of the conic set $\mathcal{R}_\pm$ in $S^*X_\delta$, noting that 
\[
L_\pm \subset \widehat{\Sigma}_\pm.
\]
The crucial observation of \cite{vasy:2013} is that $L_+$ is a source and $L_-$ a sink for the rescaled Hamilton flow on $\widehat{\Sigma}_\pm$ generated by $|\xi|^{-1}H_p$ (here $|\cdot|$ is some norm on the fibers of $T^*X$). In fact, let 
\[
\rho = |\xi_r|^{-1},
\]
which is a homogeneous degree $-1$ function defined near $\widehat{\Sigma}$. Then $H_{p} \rho$ is homogeneous of degree zero, hence a function on $S^*X_\delta$. A brief calculation gives
\[
 H_{p} \rho |_{\widehat{\Sigma}_\pm} = \pm \Delta_r'(r).
\]
Furthermore, if 
\[
p_1 = \Delta_\theta {\xi_\theta}^2 + \frac{(1-\alpha)^2}{\Delta_\theta \sin^2\theta} {\xi}_{\phi^\star}^2,
\] 
then $H_{p} p_1 = 0$. Indeed, $p_1$ is the well known Carter constant \cite{carter1968hamilton} (with the momentum dual to $t^\star$, also conserved under the geodesic flow, set to zero). Therefore
\begin{equation} \label{eq:qnm_def:escape}
 \rho H_{p} (\rho^2 p_1) |_{\widehat{\Sigma}_\pm} = \pm 2 \Delta_r'(r)\rho^2 p_1.
\end{equation}
Finally, observe that the (quadratic, nondegenerate) vanishing of $\rho_1 = \rho^2 p_1$ within $\widehat{\Sigma}_\pm$ defines $L_\pm$.

\begin{lem}  \label{lem:sourcesink}
	There exists a neighborhood $U_\pm$ of $L_\pm$ in $\widehat{\Sigma}_\pm$ such that for each $(x,\xi) \in U_\pm$,
	\[
	\exp(\mp t \rho H_{p})(x,\xi) \rightarrow L_\pm
	\]
	as $t \rightarrow \infty$.
\end{lem}

\begin{proof} As noted above, the restriction of $\rho_1$ to $\widehat{\Sigma}_\pm$ vanishes precisely on $L_\pm$. It follows from \eqref{eq:qnm_def:escape} that flow lines of $\rho H_{p}$ in a small neighborhood of $L_\pm$ within $\widehat{\Sigma}_\pm$ converge to $L_\pm$ as $\mp t \rightarrow \infty$, since $\Delta_r'(r) > 0$ near $r=r_+$.
\end{proof}

For Lemma \ref{lem:sourcesink} to be useful, one needs a global nontrapping condition implying that all integral curves starting at $\widehat \Sigma_\pm$ either tend to $L_\pm$ or otherwise reach $\{ r= r_+ -\delta\}$ in appropriate time directions.

\begin{lem} \label{lem:classicalnontrapping}
	The integral curves of $\rho H_p$ satisfy the following.
	\begin{enumerate} \itemsep6pt
		\item If $(x,\xi) \in \widehat{\Sigma}_\pm$, then $\exp(\mp t \rho H_p)(x,\xi) \rightarrow L_\pm$ as $t\rightarrow \infty$.
		\item If $(x,\xi) \in \widehat{\Sigma}_\pm \setminus L_\pm$, then there exists $T>0$ such that 
		\[
		\exp(\pm T \rho H_p)(x,\xi) \in \{r \leq r_+ -\delta\}.
		\]
	\end{enumerate}
\end{lem}
\begin{proof} \begin{inparaenum} [(1)]
		\item This statement is already implied by \eqref{eq:qnm_def:escape}, since $\Delta'_r(r) > 0$ is bounded away from zero uniformly for $r \geq r_+-\delta$.

		\item This follows from the same argument as in \cite[Section 6.3]{vasy:2013}: recall that $\widehat{\Sigma}$ is contained in $\{ \Delta_r < (1+\varepsilon) a^2 \}$, and arguing as in the latter reference,
		\[
		((1+\varepsilon)a^2 - \Delta_r) \geq \frac{\varepsilon}{1+\varepsilon} \rho_1.
		\]
		Combined with the first part, this shows that eventually $r \leq r_+-\delta$ along the flow.
	\end{inparaenum}
\end{proof}

\begin{rem}
	In the Kerr--de Sitter case, an additional restriction must be placed on $a$ to ensure that the appropriate $\Delta_r$ in that case has derivative which is bounded away from zero in the region $\{\Delta_r \leq a^2\}$, see \cite[Eq. 6.13]{vasy:2013}. This is needed to show the above nontrapping condition, which in turn is crucial to showing discreteness of QNFs. This does not present a problem for Kerr--AdS spacetimes since $\Delta'_r(r)$ is always strictly positive for $r \geq r_+ - \delta$.
\end{rem}

Recall from Section \ref{subsect:kleingordon} that $P(\lambda)^* = P(\bar \lambda)$ with respect to the measure $\varrho^{-2} A \cdot dS_t$. With this choice,
\[
\Im P(\lambda) = \frac{1}{2i}\left(P(\lambda) - P(\lambda)^*\right) \in \Psi^1(X_\delta).
\]
The homogeneous principal symbol of $\Im P(\lambda)$ is calculated from the metric by
\[\sigma_1(\Im P(\lambda))(x,\xi) = 2\left( \Im \lambda \right) \varrho^2 g^{-1}(\xi\cdot dx,dt^\star).
\]
Therefore
\[
\rho\sigma_1(\Im P(\lambda))|_{L_\pm} = \mp 2  (1-\alpha)(r_+^2+a^2)\Im \lambda = -\varkappa^{-1} (\Im \lambda) (H_{p} \rho)|_{L_\pm} ,  
\]
where $\varkappa > 0$ is the surface gravity. This factorization of the subprincipal symbol at $L_\pm$ gives a threshold value for $\Im \lambda$ in the radial point estimates of Melrose \cite{melrose1994spectral}, adapted to this setting by Vasy \cite{vasy:2013}. The following microlocal result says regularity can be propagated away from $\mathcal{R}_\pm$  provided one works with high regularity Sobolev spaces; recall here that $\delta > 0$.
\begin{prop}[{\cite[Proposition 2.3]{vasy:2013}}] \label{prop:microradial}
	Given a compactly supported $G \in \Psi^0(X_\delta)$ such that $\mathcal{R}_\pm \subseteq \ELL(G)$, there exists a compactly supported $A \in \Psi^0(X_\delta)$ such that $\mathcal{R}_\pm \subseteq \ELL(A)$ with the following properties:
	
	Suppose $u \in \mathcal{D}'(X_\delta)$ and $GP(\lambda)u \in H^{s-1}(X_\delta)$ for $s \geq m$, where $m > 1/2 - \varkappa^{-1}\Im \lambda$. If there exists $A_1 \in \Psi^0(X_\delta)$ with $\mathcal{R}_\pm \subseteq \ELL(A_1)$ such that $A_1 u \in H^m(X_\delta)$, then $Au \in H^{s}(X_\delta)$. Moreover, there exists $\chi \in C_c^\infty(X_\delta)$ such that
	\[
	\| Au \|_{H^s(X_\delta)} \leq C  \left( \| GP(\lambda)u \|_{H^{s-1}(X_\delta)} +  \| \chi u \|_{H^{-N}(X_\delta)} \right)
	\]
	for each $N$.
\end{prop}

Similarly, there is a propagation result towards $\mathcal{R}_\pm$ provided one works with sufficiently low regularity Sobolev norms, where again $\delta > 0$.

\begin{prop} [{\cite[Proposition 2.4]{vasy:2013}}]\label{prop:microradialadjoint}
	Given a compactly supported $G \in \Psi^0(X_\delta)$ such that $\mathcal{R}_\pm \subseteq \ELL(G)$, there exist compactly supported $A ,B \in \Psi^0(X_\delta)$ such that $\mathcal{R}_\pm \subseteq \ELL(A)$ and $\WF(B) \subseteq \ELL(G) \setminus \mathcal{R}_\pm$, with the following properties:
	
	Suppose $u \in \mathcal{D}'(X_\delta)$ and $GP(\lambda)u \in H^{s-1}(X_\delta), \, Bu \in H^s(X_\delta)$ for $s < 1/2 - \varkappa^{-1}\Im \lambda$. Then $Au \in H^{s}(X_\delta)$, and moreover there exists $\chi \in C_c^\infty(X_\delta)$ such that
	\[
	\|Au\|_{H^s(X_\delta)} \leq C\left( \| GP(\lambda)u \|_{H^{s-1}(X_\delta)} + \|Bu\|_{H^s(X_\delta)}  + \| \chi u \|_{H^{-N}(X_\delta)} \right)
	\]
	for each $N$.
\end{prop}

Propositions \ref{prop:microradial}, \ref{prop:microradialadjoint} can also be applied to $P(\lambda)^*$, which switches the sign of $\Im \lambda$ in the threshold conditions.

\subsection{Analysis near $H_\delta$}

The next step is to estimate $u$ near the boundary $H_\delta$ in terms of $P(\lambda)u$. This may be done by observing that $P(\lambda)$ is strictly hyperbolic with respect to the hypersurfaces $\{r = \mathrm{constant}\}$ for $r\in (r_+-2\delta,r_+)$ and $\delta > 0$ sufficiently small. 

Given $R_1 < R_2$, let $X_{(R_1,R_2)} = \{R_1 < r < R_2\}$. Define $L^2(X_{(R_1,R_2)})$ with respect to any density which is smooth on the closure of $X_{(R_1,R_2)}$, observing that the closure is compact. If $k \in \mathbb{N}$, then $H^k(X_{(R_1,R_2)})$ will denote distributions $u \in L^2(X_{(R_1,R_2)})$ such that 
\[
V_1 \cdots V_N u \in L^2(X_{(R_1,R_2)})
\]
for any collection $V_1,\ldots, V_N$ of at most $k$ smooth vector fields on the closure of $X_{(R_1,R_2)}$. Elements of $H^k(X_{(R_1,R_2)})$ are extendible in the sense of \cite[Appendix B.2]{hormanderIII:1985} --- in the notation there,
\[
H^k(X_{(R_1,R_2)}) = \OL{H}^k(X_{(R_1,R_2)}).
\]
The next result is a consequence of basic energy estimates for hyperbolic equations, see \cite[Proposition 2.13]{hintz2013semilinear}, \cite[Proposition 3.8]{vasy:2013} in this setting, as well as \cite[Theorem 23.2.1]{hormanderIII:1985}, \cite[Section 2.8]{taylor1996partial}.

\begin{prop} \label{lem:qnm_def:hyperbolic}
	Fix $r_+-2\delta < R_0 < R_1 < R_2 < r_+$, and let $u \in H^1(X_{(R_0,R_2)})$. If $u \in H^{k+1}(X_{(R_1,R_2)})$ and $P(\lambda)u \in H^{k}(X_{(R_0,R_2)})$ for some $k \in \mathbb{N}$, then $u \in H^{k+1}(X_{(R_0,R_2)})$. Furthermore,
	\[
	\| u \|_{H^{k+1}(X_{(R_0,R_1)})} \leq C\left( \| P(\lambda)u \|_{H^k(X_{(R_0,R_2)})} + \| u \|_{H^{k+1}(X_{(R_1,R_2)})} \right),
	\]
	where $C>0$ is independent of $u$.
\end{prop}

Observe that regularity can be also be propagated backwards in Proposition \ref{lem:qnm_def:hyperbolic} by considering $-P(\lambda)$. Proposition \ref{lem:qnm_def:hyperbolic} also applies to $P(\lambda)^*$.

\subsection{Energy estimates}

Energy estimates will also be used to prove that $P(\lambda)$ is invertible in the upper half-plane. Let $N_t$ denote the future-pointing unit normal to $X_\delta$ (the time orientation is determined by the timelike covector $dt^\star$). In this subsection it is important to consider $\delta \geq 0$, but to begin assume that $\delta>0$. 

Let $dS_r$ denote the induced measure on $H_\delta$ and $N_r$ be the outward-pointing unit normal to $\mathcal{H}_\delta$. Both $N_t, \, N_r$ are timelike, and they lie in the same lightcone over $\mathcal{H}_\delta$. Recall that
\[
dg = A\cdot  dt^\star\, dS_t,
\]
where $A = g^{-1}(dt^\star,dt^\star)^{-1/2}$ and $dg$ is the volume measure. Also, if $k$ denotes the induced (Riemannian) metric on the spacelike hypersurface $\mathcal H_\delta$, let $A_r = k^{-1}(dt^\star,dt^\star)^{-1/2}$.

If $V$ is a $C^1$ vector field on $\mathcal{M} \cup \mathcal{H}$ vanishing near $\scri$, then differentiating the divergence theorem at $t^\star = 0$ gives the identity
\begin{equation} \label{eq:divergence0}
\frac{d}{dt^\star} \int_{X_\delta} g(V,N_t)  \, dS_t  + \int_{H_\delta} g(V,N_r)\, A_r\, dS_r= \int_{X_\delta} \left(\mathrm{div}_{g} V \right) A\,dS_t.
\end{equation}

Now suppose that $\delta =0$. In that case the hypersurface $\mathcal{H}_0 = \mathcal{H}^+$ is null, and hence $N_r$ is ill-defined. Nevertheless, setting $N_r = K$ as in \eqref{eq:generator} and $A_r=1$, the equality \eqref{eq:divergence0} still holds. Note that $dS_r$ is always well defined since $H_\delta \subseteq X_\delta$ and $X_\delta$ is spacelike.

Given a $C^2$ function $v$ on ${\mathcal{M}}_\delta \cup \mathcal{H}_\delta$, the stress-energy tensor $\mathbb{T} = \mathbb{T}[v]$ associated to the wave equation is 
\[
\mathbb{T}(Y,Z) =\Re \left( Yv \cdot Z\bar{v} \right) - \tfrac{1}{2}g(Y,Z) g^{-1}(dv,d\bar {v}).
\]
Here $Y,Z$ are real $C^1$ vector fields on $\mathcal{M}_\delta \cup \mathcal{H}_\delta$. It is well known that $\mathbb{T}(Y,Z)$ is nonnegative if $Y, Z$ are causal (timelike or null) in the same lightcone, and positive definite in $dv$ if both $Y, Z$ are timelike \cite[Lemma 24.1.2]{hormanderIII:1985}. 

Let $\mathbb{J}^Y = \mathbb{J}^Y[v]$ be the unique vector field such that $g(\mathbb{J}^Y,Z) = \mathbb{T}(Y,Z)$. If $F =  (\Box_g + \nu^2 - 9/4)v$, then  
\begin{equation} 
\mathrm{div}_{g}\, \mathbb{J}^Y = \Re \left( F \cdot Y\bar{v} \right) + Q,
\end{equation}
where $Q$ is a real quadratic form in $(v,dv)$. Apply \eqref{eq:divergence0} to the vector field $\mathbb{J}^Y$, where $v$ vanishes for $r$ sufficiently large. This yields the identity
\begin{equation} \label{eq:qnm_def:divergence} 
\frac{d}{dt^\star} \int_{X_\delta} \mathbb{T}(Y,N_t)\, dS_t + \int_{H_\delta} \mathbb{T}(Y, N_r) \,A_r \, dS_r = 
\int_{X_\delta} \left( \Re \left( F \cdot Y\bar{v} \right) + Q \right)A \,dS_t.
\end{equation}

Now suppose that $Y,Z$ are stationary in the sense that $\mathcal{L}_{\partial_{t^\star}} Y= \mathcal{L}_{\partial_{t^\star}} Z = 0$. Given a function $u$ on $X_\delta$, let $v= e^{-i\lambda t^\star} u$, viewed as a function on $\mathcal{M}_\delta$. Then, the stress-energy tensor associated to $v = e^{-i\lambda t^\star}u$ satisfies
\[
\frac{d}{dt^\star} \mathbb{T}[v](Y,Z) = 2(\Im \lambda)\mathbb{T}[v](Y,Z).
\]
Furthermore, if the stationary function $e^{-2(\Im\lambda) t^\star}\,\mathbb{T}[v](Y,Z)$ is viewed as a function on $X_\delta$, then it is a positive definite quadratic form in $(du,\lambda u)$, where now $du$ is the differential of $u$ on $X_\delta$.

On the other hand, if $v = e^{-i \lambda t^\star}u$, then for $t^\star =0$ the integrand on the right hand side of \eqref{eq:qnm_def:divergence} can be written as
\[
\varrho^{-2}\Re \big( P(\lambda)u \cdot \OL{Y(\lambda)u} \,\big)   +  Q(\lambda),
\]  
where $Q(\lambda)$ is a quadratic form in $(du,u,\lambda u)$, and $Y(\lambda)u = e^{i\lambda t^\star}Y(e^{-i\lambda t^\star}u)$.

%
%
%
%
%
%
%

\begin{lem} \label{lem:qnm_def:invertible1}
	Fix $\delta \geq 0,\, R > r_+$, and let $u \in C^2_c(\{ r_+ - \delta \leq r < R \})$. There exists $C_0>0$ such that
	\[
	|\lambda| \| u \|_{\mathcal{L}^2(X_\delta)} + \| du \|_{\mathcal{L}^2(X_\delta)} \leq \frac{C}{\Im \lambda} \| P(\lambda) u \|_{\mathcal{L}^2(X_\delta)}
	\]
	for $\Im \lambda > C_0$, where $C>0$ is independent of $\lambda$ and $u$.
\end{lem}
\begin{proof}
	Apply \eqref{eq:qnm_def:divergence} with the multiplier $Y = N_t$ and $v = e^{-i\lambda t^\star}u$, recalling that all terms are evaluated at $t^\star = 0$. First, observe that the integral over $H_\delta$ is nonnegative, since $N_r$ and $N_t$ are both in the same lightcone (of course $N_r = K$ is null if $\delta = 0$).
	With $f = P(\lambda)u$,
	\[
	\Im \lambda \left( |\lambda|^2 \| u \|_{\mathcal L^2(X_\delta)}^2  + \| du \|_{\mathcal L^2(X_\delta)}^2 \right) \leq  C\int_{X_\delta} \big(  \varrho^{-2} \Re \big(  f  \cdot \OL{N_t(\lambda) u} \big) + Q(\lambda) \,\big) A \, dS_t.
	\]
	Both $A$ and $\varrho^{-2}$ are bounded by constants depending on $R$. Furthermore, the quadratic form $Q(\lambda)$ can be absorbed into the left hand side for $\Im \lambda >0$ sufficiently large. The integrand involving $f$ is bounded by Cauchy--Schwarz, yielding
	\[
	\Im \lambda \left( |\lambda|^2 \| u \|_{\mathcal L^2(X_\delta)}^2  + \| du \|_{\mathcal L^2(X_\delta)}^2 \right) \leq \frac{C}{\Im \lambda} \| f \|^2_{\mathcal L^2(X_\delta)}
	\]
	as desired.\end{proof}

A similar argument applies to $P(\lambda)^*$ provided $u$ vanishes along $H_\delta$.

\begin{lem} \label{lem:qnm_def:invertible2}
Fix $\delta \geq 0,\, R> r_+$, and let $u \in C^2_c(\{r_+ - \delta \leq r < R\})$ be such that $u|_{H_\delta} = 0$. There exists $C_0>0$ such that
	\[
	|\lambda| \| u \|_{\mathcal L^2(X_\delta)} + \| du \|_{\mathcal L^2(X_\delta)} \leq \frac{C}{\Im \lambda} \| P(\lambda)^* u \|_{\mathcal L^2(X_\delta)},
	\]
	for $\Im \lambda > C_0$, where $C>0$ is independent of $\lambda$ and $u$.
\end{lem}
\begin{proof}
	Since $P(\lambda)^* = P(\bar \lambda)$, apply \eqref{eq:qnm_def:divergence} to $v = e^{-i\bar\lambda t^\star}u$ with the multiplier $Y = N_t$; the difference is that now the two integrals on the left hand side of \eqref{eq:qnm_def:divergence} have opposite signs for $\Im \lambda > 0$. However, if $u$ vanishes at $H_\delta$, then the same argument as in Lemma \ref{lem:qnm_def:invertible1} applies.
\end{proof}

\section{The anti-de Sitter end}  \label{sect:adsend}

This section concerns the analysis near $Y$, hence does not depend on any extension of the metric across the horizon. After a conjugation by $r$, the rescaled stationary Klein--Gordon operator $P(\lambda)$
 is a parameter-dependent Bessel operator in the sense of Section \ref{subsect:bessel}:
 
\begin{lem} \label{lem:adsiselliptic}
	$r P(\lambda) r^{-1}$ is a Bessel operator of order $\nu$ near $Y$. Furthermore, $r P(\lambda) r^{-1}$ is parameter-elliptic with respect to any angular sector $\Lambda \subset \mathbb{C}$ disjoint from $\mathbb{R}\setminus 0$.
\end{lem}
\begin{proof}
	Observe from \eqref{eq:dualmetric} that for $r$ sufficiently large, $dr$ is orthogonal to the span of $\{d{t^\star}, d\theta,d{\phi^\star}\}$. Therefore the only term in $\Box_g$ involving $r$-derivatives is
	\[
	\varrho^{-2} D_r \left( \Delta_r D_r \right).
	\]
	The remaining terms in $\Box_g$ are smooth up to $\scri$ after multiplication by $\varrho^{2}$. In the notation in Section \ref{subsubsect:besseldefs}, let $x'$ be local coordinates on $Y$ and $x_n = s$. From \eqref{eq:evenasymptotics} it is verified that $r P(\lambda) r^{-1}$ can locally be written in the form \eqref{eq:besseloperator}.

The parameter-ellipticity of $rP(\lambda)r^{-1}$ at $Y$ follows from the timelike nature of $\partial_{t^\star}$ and $dt^\star$ at $\scri$ with respect to the conformal metric $s^2 g$, using the same argument as in Lemma \ref{lem:lorentzianmetrics}.
	\end{proof}
	
	Conjugation by $r^{-1}$ corresponds to working with $r \mathcal{L}^2(X_\delta)$ based spaces.  Note that the measure defining $r \mathcal{L}^2(X_\delta)$ is locally equivalent near $Y$ to ordinary Lebesgue measure, agreeing with the convention in Section \ref{subsubsect:besseldefs}. Henceforth $P(\lambda)$ will be considered instead of $r P(\lambda) r^{-1}$, making sure to account for the additional conjugation.

When $0 < \nu < 1$, the operator $P(\lambda)$ must be augmented by elliptic boundary conditions as in Section \ref{subsubsect:boundaryoperators}. Thus assume that $T(\lambda)$ is a parameter-dependent boundary operator of the form
\[
T(\lambda) = (T_1^- + \lambda T_0^-)\gamma_- + T_1^+ \gamma_+,
\]
where the weighted restrictions $\gamma_\pm$ are given by
\[
\gamma_- u = s^{\nu-3/2} u|_{Y}, \quad \gamma_+ u = -s^{1-2\nu} \partial_s(s^{\nu-3/2} u)|_{Y}.
\]
Here $\gamma_\pm$ are redefined from Section \ref{subsubsect:boundaryoperators} to account for the conjugation by $r^{-1}$. It is assumed that the ``principal part'' of $T(\lambda)$ (in the sense of Section \ref{subsubsect:boundaryoperators}) is independent of $\lambda$. Ellipticity and parameter-ellipticity of the operator
\[
\mathscr{P}(\lambda) = \begin{pmatrix} P(\lambda) \\ T(\lambda) \end{pmatrix}
\]
with respect to $\Lambda$ were defined in Section \ref{subsubsect:lopatinskii}. 

\subsection{Function spaces} \label{subsect:adsfunctionspaces}

Following \cite[Section 4]{gannot:bessel}, ellipticity is used to prove coercive estimates for functions supported near $Y$. These local estimates should be understood as comprising part of a global estimate. For this reason, it is useful to state them on function spaces which are globally defined on $X_\delta$. These spaces are now described.

Let $\Sob^1(X_\delta)$ denote the set of all distributions $u \in \mathcal{L}^2(X_\delta)$ such that $s^{3/2-\nu}d(s^{\nu-3/2} u) \in \mathcal{L}^2(X_\delta)$, where the magnitude of a covector is measured with respect to a smooth norm on $\OL{X}_\delta$. Set 
\[
\| u \|_{\Sob^1(X_\delta)} = \| u \|_{\mathcal{L}^2(X_\delta)} + \| s^{3/2-\nu}d(s^{\nu-3/2} u) \|_{\mathcal{L}^2(X_\delta)}.
\]
 To define higher order spaces, let $\mathcal{V}_b(\OL{X}_\delta)$ denote the space of smooth vector fields on $\OL{X}_\delta$ which are tangent to $Y$ (but \emph{not} necessarily to $H_\delta$). Given $k\in \mathbb{N}$ and $s=0,1$, let $\Sob^{s,k}(X_\delta)$ denote the set of distributions $u$ such that 
 \[
 V_1 \cdots V_N u \in \Sob^{s}(X_\delta)
 \] 
 for any collection $V_1,\ldots,V_N$ of at most $k$ vector fields in $\mathcal{V}_b(\OL{X}_\delta)$. These spaces can be normed in the obvious way by fixing a finite generating set of vector fields for $\mathcal{V}_b(\OL{X}_\delta)$. Over any compact subset of $X_\delta$ the norms of $\Sob^{s,k}(X_\delta)$ and $H^{s+k}(X_\delta)$  are equivalent.

If $0 < \nu < 1$, let ${\mathcal{F}}_\nu(X_\delta)$ denote the space of $u \in C^\infty(X_\delta\cup H_\delta)$ which near $Y$ have the form
\[
s^{3/2-\nu}u_-(s^2,y) + s^{3/2+\nu}u_+(s^2,y)
\]
for $u_\pm \in C^\infty([0,\varepsilon)_s\times Y)$. If $\nu \geq 1$ then ${\mathcal{F}}_\nu(X_\delta)$ is defined to be $C_c^\infty(X_\delta\cup H_\delta)$. In both cases $\mathcal{F}_\nu(X_\delta)$ is dense in $\Sob^{s,k}(X_\delta)$ \cite[Section 3]{gannot:bessel}.

\begin{rem} Duality for these spaces is not described here; a detailed discussion, including everything needed for this paper, can be found in \cite[Sections 3, 4, 5]{gannot:bessel}.
	\end{rem}

\subsection{Elliptic estimates} \label{subsect:adsellipticestimate}

The results in this section follow from \cite[Theorems 1, 2, 3]{gannot:bessel}. First assume that $\nu \geq 1$. According to Lemma \ref{lem:adsiselliptic}, $P(\lambda)$ is elliptic at $Y$, and parameter-elliptic at $Y$ with respect to any angular sector $\Lambda \subset \mathbb{C}$ disjoint from $\mathbb{R}\setminus 0$. If $0 < \nu < 1$, then ellipticity and parameter-ellipticity must be assumed for $\mathscr{P}(\lambda)$ with respect to $\Lambda$.

\begin{prop}[{\cite[Theorems 1, 3]{gannot:bessel}}] \label{prop:nugeq1}
	Let $k \in \mathbb{N}$. There exists $\varepsilon > 0$ such that if $\varphi, \chi \in C_c^\infty(\{0 \leq s < \varepsilon\})$ satisfy $\varphi = 1$ near $s=0$ and $\chi = 1$ near $\supp \varphi$, then the following hold:
	\begin{enumerate}  \itemsep6pt
		\item If $\nu \geq 1$, then there exists $C>0$ such that
		\[
		\|  \varphi u \|_{\Sob^{1,k}(X_\delta)} \leq C \left( \| \chi P(\lambda)u \|_{\Sob^{0,k}(X_\delta)} + \| \chi u \|_{\Sob^{0}(X_\delta)} \right)
		\]
		for each $u \in \mathcal{F}_\nu(X_\delta)$.
		
		\item If $0 < \nu < 1$ and $\mathscr{P}(\lambda)$ is elliptic at $Y$, then there exists $C>0$ such that
		\begin{align*}
		\|  \varphi u \|_{\Sob^{1,k}(X_\delta)}  \leq  C (	\| \chi \mathscr{P}(\lambda)u \|_{\Sob^{0,k}(X_\delta)\times H^{k+2-\mu}(Y)}
		+ \| \chi u \|_{\Sob^{0}(X_\delta)} )
		\end{align*}
		for each $u \in \mathcal{F}_\nu(X_\delta)$.
	\end{enumerate}
\end{prop}

\noindent There is also a regularity statement associated with Proposition \ref{prop:nugeq1}, namely if $u \in \Sob^0(X_\delta)$ and the right-hand sides are finite, then so are the left-hand sides. Making sense of this when $0 < \nu < 1$ is slightly subtle, and the reader is again referred to \cite{gannot:bessel} for details. 

To obtain estimates which are uniform $\lambda$, parameter-ellipticity is used. These estimate are only used with $k=0$.

\begin{prop} [{\cite[Theorem 2]{gannot:bessel}}]\label{prop:0<nu<1}
	Fix an angular sector $\Lambda \subset \mathbb{C}$ such that $P(\lambda)$ and $\mathscr{P}(\lambda)$ are parameter elliptic at $Y$ with respect to $\Lambda$. There exists $\varepsilon > 0$ such that if $\varphi, \chi \in C_c^\infty(\{0 \leq s < \varepsilon\})$ satisfy $\varphi = 1$ near $s=0$ and $\chi = 1$ near $\supp \varphi$, then the following hold:
	\begin{enumerate}  \itemsep6pt
		\item If $\nu \geq 1$, then there exists $C>0$ such that
		\[
		|\lambda| \| \varphi u \|_{\Sob^0(X_\delta)} + \|\varphi u \|_{\Sob^{1}(X_\delta)} \leq C \left( \| \chi P(\lambda)u \|_{\Sob^{0}(X_\delta)} + \| \chi u \|_{\Sob^{0}(X_\delta)} \right)
		\]
		for each $u \in \mathcal{F}_\nu(X)$ and $\lambda \in \Lambda$.
		
		\item If $0 < \nu < 1$, then there exists $C>0$ such that
		\begin{align*}
		|\lambda| \| \varphi u \|_{\Sob^0(X_\delta)} +\|  \varphi u \|_{\Sob^{1}(X_\delta)}  \leq  C (	\| \chi \mathscr{P}(\lambda)u \|_{\Sob^{0}(X_\delta)\times H^{2-\mu}(Y)}
		+ \| \chi u \|_{\Sob^{0}(X_\delta)} )
		\end{align*}
		for each $u \in \mathcal{F}_\nu(X_\delta)$ and $\lambda \in \Lambda$.
	\end{enumerate}
\end{prop}

Estimates for the formal adjoint $P(\lambda)^*$ if $\nu \geq1$, or $\mathscr{P}(\lambda)^*$ if $0<\nu < 1$, also hold. However, the formal adjoint $\mathscr{P}(\lambda)^*$ is no longer a scalar operator --- see \cite[Section 4]{gannot:bessel} where the formal adjoint is defined (which is entirely analogous to the adjoint in the Boutet de Monvel calculus for smooth boundary value problems \cite{de1971boundary}). Furthermore, \cite[Theorem 1]{gannot:bessel} only treats estimates for the formal adjoint when $k=0$, although this does not present a problem here.

\section{Fredholm property and meromorphy} \label{sect:fredholm}

In this section the Fredholm property for $P(\lambda)$ and meromorphy of $P(\lambda)^{-1}$ are derived from estimates on $P(\lambda)$, combined with some standard arguments from functional analysis. Of course $P(\lambda)$ should be replaced by $\mathscr{P}(\lambda)$ when $0 < \nu < 1$. For $\delta \geq0$, introduce the space
\[
\mathcal{X}^k(X_\delta) = \{u \in {\Sob}^{1,k}(X_\delta) : P(0) u \in {\Sob}^{0,k}(X_\delta) \},
\]
equipped with the norm $\| u \|_{{\Sob}^{1,k}(X_\delta)} + \| P(0)u \|_{{\Sob}^{0,k}(X_\delta)}$. This space is complete, and in fact ${\mathcal{F}}_\nu(X_\delta)$ is dense in $\mathcal{X}^k(X_\delta)$ \cite[Lemma 5.1]{gannot:bessel}.

\subsection{The case $\nu \geq 1$}

Fix $\delta > 0$ and consider the simpler case $\nu \geq 1$ first. Initially, the goal is to prove that 
\[
P(\lambda) : \mathcal{X}^k(X_\delta) \rightarrow \Sob^{0,k}(X_\delta)
\] 
has closed range and finite dimensional kernel for each $k \in \mathbb{N}$, provided $\lambda$ lies in an appropriate half-plane.

\begin{prop} \label{prop:nugeq1P}
	If $C_0 < \varkappa(k+1/2)$, then there exists $\varphi \in C_c^\infty(X_\delta \cup Y)$ and $\chi \in C_c^\infty(X_\delta)$ such that
	\begin{equation} \label{eq:nugeq1globalelliptic}
	\| u \|_{{\Sob}^{1,k}(X_\delta)} \leq C \left( \| P(\lambda)u \|_{{\Sob}^{0,k}(X_\delta)} + \| \chi u \|_{H^{-N}(X_\delta)} + \| \varphi u \|_{ {\Sob}^{0}(X_\delta) } \right)
	\end{equation}
	for any $N$ and $u \in \mathcal{F}_\nu(X_\delta)$, provided $\Im \lambda > -C_0$. 
\end{prop} 

\begin{proof}
Begin by choosing two functions $\zeta, \psi \in C^\infty(\OL X_\delta; [0,1])$ subject to the following conditions:
	\begin{enumerate} \itemsep6pt
		\item $\supp \psi \subseteq \{0 \leq s < \varepsilon \}$ and $\psi =1$ near $s=0$, where $\varepsilon>0$ is sufficiently small.
		\item $\supp \zeta \subseteq \{ r_+-\delta \leq r < r_+- 2\delta/3 \}$ and $\zeta = 1$ near $\{ r_+-\delta \leq r < r_+-3\delta/4 \}$.
	\end{enumerate}
		Let $u \in {\mathcal{F}}_\nu(X_\delta)$ and $f = P(\lambda)u$. It is possible to find a microlocal partition of unity
	\[
	1 = \zeta + \psi + \sum_{j=1}^J A_j + R,
	\]
	where the operators $A_j \in \Psi^0(X_\delta), \ R \in \Psi^{-\infty}(X_\delta)$ are compactly supported, and each $A \in \{A_1,\ldots, A_J\}$ has one of the following properties:
	
	\begin{enumerate} \itemsep6pt
		\item $\WF(A) \subseteq \ELL(P(\lambda))$. By microlocal elliptic regularity (Proposition \ref{prop:microelliptic}),
		\[
		\| A u \|_{H^{s+1}(X_\delta)} \leq C \| G f \|_{H^{s-1}(X_\delta)} + \| \chi u \|_{H^{-N}(X_\delta)}
		\] 
		for $G$ microlocalized near $\WF(A)$ and some $\chi \in C_c^\infty(X_\delta)$.
		
		\item \label{it:sourcesink} $\WF(A)$ is contained in a small conic neighborhood of $\mathcal{R}_\pm$. In order to apply Proposition \ref{prop:microradial}, the imaginary part of $\lambda$ must satisfy $\Im \lambda \geq -C_0$ for some $C_0 <  \varkappa (s-1/2)$. In that case,
		\[
		\| A u \|_{H^s(X_\delta)} \leq C \left( \| G f \|_{H^{s-1}(X_\delta)} + \| \chi u \|_{H^{-N}(X_\delta)} \right)
		\] 
		for some $G$ microlocalized near $\WF(A)$ and some $\chi \in C_c^\infty(X_\delta)$.
		
		\item $\WF(A)$ is contained in a conic neighborhood of a point $(x_0,\xi_0) \in {\Sigma}_+ \setminus \mathcal{R}_+$. Then there is a conic neighborhood $U_+ \supseteq \mathcal{R}_+$ such that for each $B\in\Psi^0(X_\delta)$ with $\WF(B) \subseteq U_+$ and $(x,\xi) \in \WF(A)$, there exists $T>0$ with
		\[
		\exp(-T H_{p})(x,\xi) \subseteq \ELL(B).
		\] 
	This follows from Lemma \ref{lem:classicalnontrapping}, shrinking $\WF(A)$ if necessary. It is now possible to combine propagation of singularities (\cite[Section 2.3]{vasy:2013}) with the previous item \eqref{it:sourcesink}. For some $G_1$ microlocalized near the union of flow lines emanating from $\WF(A)$ and $G$ as in \eqref{it:sourcesink},
		\[
		\| Au \|_{H^s(X_\delta)} \leq C\left( \| Gf \|_{H^{s-1}(X_\delta)} + \| G_1 f\|_{H^{s-1}(X_\delta)} +  \| \chi u \|_{H^{-N}(X_\delta)} \right) 
		\]
		for some $\chi \in C_c^\infty(X_\delta)$.
		The same argument applies if $(x_0,\xi_0) \in {\Sigma}_- \setminus \mathcal{R}_-$, reversing the direction of propagation.
	\end{enumerate}
The estimates on $Au$ are applied with Sobolev index $s = 1+k$ where $k \in \mathbb{N}$, which gives $C_0 < \varkappa(k+1/2)$. The term $\psi u$ is then estimated in ${\Sob}^{1,k}(X_\delta)$ using Proposition \ref{prop:nugeq1}, provided $\supp \psi$ is sufficiently small. In the region where $r < r_+$, apply Lemma \ref{lem:qnm_def:hyperbolic}:
	\[
	\|\zeta u \|_{{H}^{k+1}(X_\delta)} \leq  C \left(  \| P(\lambda)u \|_{\Sob^{0,k}(X_\delta)} + \|\zeta' u \|_{H^{k+1}(X_\delta)}\right),
	\]
	where $\zeta'$ has compact support in $\{ r_+ -\delta/2 < r < r_+\}$. In particular, 
	\[
	\zeta' \zeta = \zeta' \psi = 0,
	\]
	and hence $A_1 + \cdots + A_J$ is elliptic on $\supp \zeta'$ (lifted to $T^*X_\delta$). Therefore $\zeta' u$ is controlled by the $A_j u$ terms handled above.
\end{proof}

Although \eqref{eq:nugeq1globalelliptic} of Proposition \ref{prop:nugeq1P} is stated as an a priori estimate (namely $u$ is assumed to lie in $\mathcal{F}_\nu(X_\delta)$), the proof also gives the following regularity result:

\begin{prop} \label{prop:regularity}
If $\Im \lambda > -\varkappa(k+1/2)$ and $u \in \Sob^{1,k}(X_\delta)$ satisfies $P(\lambda)u \in \Sob^{0,k'}(X_\delta)$ for $k' \geq k$, then $u \in \mathcal{X}^{k'}(X_\delta)$.
\end{prop}

The multiplication maps $\varphi : \Sob^{1,k}(X_\delta) \rightarrow \Sob^{0}(X_\delta)$ and $\chi: \Sob^{1,k}(X_\delta) \rightarrow H^{-N}(X_\delta)$
are compact provided $N$ is sufficiently large; in the former case, compactness comes from \cite[Lemma 3.21]{gannot:bessel}. Then the a priori estimate \eqref{eq:nugeq1globalelliptic} shows that $P(\lambda) : \mathcal{X}^k(X_\delta) \rightarrow \Sob^{0,k}(X_\delta)$ has closed range and finite-dimensional kernel for $\Im \lambda  > -\varkappa(k+1/2)$.

\begin{lem} \label{lem:qnm_def:L2invertible}
	Fix an angular sector $\Lambda$ in the upper half-plane which is disjoint from $\mathbb{R}\setminus 0$. Then there exists $R>0$ such that  $P(\lambda) : \mathcal{X}^0(X_\delta) \rightarrow \Sob^0(X_\delta)$ is invertible for $\lambda \in \Lambda$ and $|\lambda|>R$.
\end{lem}
\begin{proof}
	\begin{inparaenum}[(1)] \item If $\Im \lambda > -\varkappa/2$, then $u \in \mathcal{X}^0(X_\delta)$ and $P(\lambda)u=0$ together imply that $u\in C^\infty(X_\delta \cup H_\delta)$ --- this follows from Proposition \ref{prop:regularity}. Next, fix $\varphi \in C^\infty_c(X_\delta\cup H_\delta)$. 
		Since $u$ is smooth and $\varphi u $ has support in a fixed compact set, from Lemma \ref{lem:qnm_def:invertible1} there exist $C_0>0$
		\[
		|\lambda| \| \varphi u \|_{\mathcal{L}^2(X_\delta)} + \| \varphi u\|_{\Sob^1(X_\delta)}   \leq \frac{C}{\Im \lambda}\|P(\lambda)(\varphi u)\|_{\mathcal L^2(X_\delta)}
		\]
		for $\Im \lambda > C_0$, where $C>0$ depends only on the support of $\varphi$. Note that $P(\lambda)\varphi u = [P(\lambda),\varphi]u$, which is therefore estimated by
		\[
		\| [P(\lambda),\varphi]u \|_{\mathcal{L}^2(X_\delta)} \leq C\left( \|u\|_{\Sob^1(X_\delta)} + |\lambda|\| u \|_{\mathcal{L}^2(X_\delta)}\right).
		\]
		On the other hand, if $\supp \varphi$ is sufficiently large, then from Proposition \ref{prop:0<nu<1},
		\[
		|\lambda| \| (1-\varphi) u \|_{\mathcal{L}^2(X_\delta)}  + \| (1-\varphi) u \|_{\Sob^1(X_\delta)} \leq C	\| u \|_{\mathcal{L}^2(X_\delta)}
		\]
		for $\lambda \in \Lambda$.  Combining the estimates for $\varphi u$ and $(1-\varphi)u$ shows that
		\begin{align*}
		|\lambda|\| u \|_{\mathcal{L}^2(X_\delta)} + \| u \|_{\Sob^1(X_\delta)} &\leq {C}\left(|\lambda|^{-1} + (\Im \lambda)^{-1} \right) |\lambda| \| u \|_{\mathcal{L}^2(X_\delta)}  \\ &+ C\left(\Im \lambda\right)^{-1} \| u \|_{\Sob^1(X_\delta)}
		%
		\end{align*}
		provided $\Im \lambda > C_0$ and $\lambda \in \Lambda$. Since $\Lambda$ is contained in the upper half-plane, $\lambda \in \Lambda$ and $|\lambda|$ large imply that $u= 0$. Therefore $P(\lambda)$ is injective in this region.

		\item To prove that $P(\lambda)$ is surjective, it suffices to prove that the formal adjoint $P(\lambda)^*$ is injective on $\mathcal{L}^2(X_\delta)$. For duality purposes, $P(\lambda)^*$ acts on $\mathcal{L}^2(X_\delta)$ in the sense of distributions supported on $X_\delta\cup H_\delta$, see \cite[Appendix B.2]{hormanderIII:1985}. Therefore $P(\lambda)^*v = 0$ means that
		\begin{equation} \label{eq:duality}
		\int_{X_\delta} \big( \,\OL{P(\lambda)\phi} \cdot v \big) \, \varrho^{-2}A \,dS_t = 0
		\end{equation}
		for each $\phi \in \mathcal{F}_\nu(X_\delta)$. Extend $v$ by zero to $v_1 \in \mathcal{L}^2(X_{2\delta})$.  Now $P(\lambda)^*$ is still defined on $X_{2\delta}$, and $P(\lambda)^*  v_1 =0$ in the sense of distributions on $X_{2\delta}$ since any $\phi \in C_c^\infty(X_{2\delta})$ can be restricted to an element of $\mathcal{F}_\nu(X_\delta)$. Since 
		\[
		\supp v_1 \subseteq \{r \geq r_+-\delta \},
		\] 
		$v_1$ is smooth on $\{ r_+ -2\delta < r < r_+\}$ by Lemma \ref{lem:classicalnontrapping} and propagation of singularities. Therefore $v_1 = 0$ on $\{ r_+ -2\delta < r < r_+\}$ by Proposition \ref{lem:qnm_def:hyperbolic}. The same argument as in Proposition \ref{prop:nugeq1P} now shows that $v \in C^k(X_\delta\cup H_\delta)$ provided $\Im \lambda > 0$ is sufficiently large depending on $k$ --- this involves replacing $P(\lambda)$ with $P(\lambda)^*$ and then using Proposition \ref{prop:microradialadjoint} instead of Proposition \ref{prop:microradial}. Furthermore, $v \in \Sob^{1,k}(X_\delta)$ near $Y$ for arbitrary $k \in \mathbb{N}$ according to \cite[Theorem 3]{gannot:bessel}. Now the same argument for $P(\lambda)$ applies to show that $P(\lambda)^*$ is injective, using that $v$ vanishes at $H_\delta$ in order to use Lemma \ref{lem:qnm_def:invertible2}. \end{inparaenum} \end{proof}

\noindent  It is now possible to prove Theorem \ref{theo:mainextended} for the case $\nu \geq 1$.
\begin{proof} [Proof of Theorem \ref{theo:mainextended} for $\nu \geq 1$] Given $k \in \mathbb{N}$, write $P^{(k)}(\lambda)$ for the operator 
	\[
	P(\lambda) : \mathcal{X}^k(X_\delta) \rightarrow {\Sob}^{0,k}(X_\delta).
	\]
	Proposition \ref{prop:nugeq1P} shows that $P^{(k)}(\lambda)$ has closed range and finite dimensional kernel in the half-plane $\Im \lambda > -\varkappa(k+1/2)$. According to Lemma \ref{lem:qnm_def:L2invertible}, there exists $\lambda_0$ with sufficiently large imaginary part so that $P^{(0)}(\lambda_0)$ is invertible. Clearly injectivity of $P^{(0)}(\lambda_0)$ implies injectivity of $P^{(k)}(\lambda_0)$. Furthermore, suppose that $f \in {\Sob}^{0,k}(X_\delta) \subseteq {\Sob}^{0}(X_\delta)$. If $u \in \mathcal{X}^0(X_\delta)$ denotes the unique solution to
	\[
	P^{(0)}(\lambda_0)u = f,
	\]
	then $u \in \mathcal{X}^k(X_\delta)$ since $\Im \lambda_0 > -\varkappa/2$. Thus $P^{(k)}(\lambda)$ is invertible in the upper-half plane wherever $P^{(0)}(\lambda)$ is invertible. Furthermore, $P^{(k)}(\lambda)$ is Fredholm of index zero in the half-plane $\Im \lambda > -\varkappa(k+1/2)$, since the index of a continuous family of left semi-Fredholm operators (namely those with closed range and finite dimensional kernel) is constant on connected components \cite[Theorem 19.1.5]{hormanderIII:1985}.\end{proof}

\subsection{The case $0 < \nu < 1$}

Fix a boundary operator $T(\lambda)$ as in Section \ref{sect:adsend} such that
\[
\mathscr{P}(\lambda) = \begin{pmatrix} P(\lambda) \\ T(\lambda) \end{pmatrix}
\]
is elliptic with respect to an angular sector $\Lambda \subseteq \mathbb{C}$ disjoint from $\mathbb{R}\setminus 0$. Assume that the principal part of $T(\lambda)$ is independent of $\lambda$.

\begin{proof} [Proof of Theorem \ref{theo:mainextended} for $0< \nu < 1$] Proposition \ref{prop:nugeq1P} has a natural analogue in this setting: the microlocal estimates on $X_\delta$ and hyperbolic estimates near $H_\delta$ are unchanged. Near $Y$ apply Proposition \ref{prop:nugeq1} for the case $0< \nu < 1$, referring to \cite[Sections 5.1, 5.2]{gannot:bessel} to see how the condition $T(\lambda)u \in H^{k+2-\mu}(Y)$ is used in general. Invertibility of $\mathscr{P}(\lambda)$ for $k=0$ follows as in Lemma \ref{lem:qnm_def:L2invertible}; the analysis of the adjoint problem is slightly more involved, see \cite[Section 5.2]{gannot:bessel}. The same argument as in the proof for $\nu \geq 1$ handles larger values of $k$.
\end{proof}

\subsection{Passing from $X_\delta$ to $X_0$} \label{subsect:passing}

Equipped with Theorem \ref{theo:mainextended}, it is now possible to deduce Theorems \ref{theo:maintheo1}, \ref{theo:maintheo2} as well. 

\begin{proof}[Proof of Theorems \ref{theo:maintheo1}, \ref{theo:maintheo2}]
	It suffices to prove Theorem \ref{theo:maintheo1} (namely when $\nu \geq 1$), since the proof of Theorem \ref{theo:maintheo2} (when $0 < \nu < 1$) is identical upon replacing $P(\lambda)$ with $\mathscr{P}(\lambda)$. Let $\Im \lambda > -\varkappa(k +1/2)$. Since 
	\[
	P(\lambda): \mathcal{X}^k(X_\delta) \rightarrow \Sob^{0,k}(X_\delta)
	\]
	has finite codimensional range for $\delta>0$, so does $P(\lambda)|_{\mathcal{X}^k(X_0)}$. To see this, fix a continuous extension map $E_k: \Sob^{0,k}(X_0) \rightarrow \Sob^{0,k}(X_\delta)$. It is then clear that $\ran E_k$ contains a subspace 
	\[
	S = \ran E_k \cap \ran P(\lambda)|_{\mathcal{X}^k(X_\delta)}  \subseteq \ran P(\lambda)|_{\mathcal{X}^k(X_\delta)}
	\] 
	which has finite codimension in $\ran E_k$. Now $E_k$ is injective, so $E_k^{-1}(S)$ has finite codimension in $\Sob^{0,k}(X_0)$. But if $f \in E_k^{-1}(S)$, then the equation
	\[
	P(\lambda)\tilde{u} = E_k f
	\]
	has a solution $\tilde{u} \in \mathcal{X}^k(X_\delta)$. Restricting $\tilde{u}$ to $X_0$ shows that $E_k^{-1}(S)$ is contained in $\ran P(\lambda)|_{\mathcal{X}^k(X_0)}$, hence the latter also has finite codimension in $\Sob^{0,k}(X_0)$; as the image of a continuous map, it is also closed.
	Therefore
		\[
	P(\lambda): \mathcal{X}^k(X_0) \rightarrow \Sob^{0,k}(X_0)
	\]
	is an analytic family of right semi-Fredholm operators in $\Im \lambda > -\varkappa(k+1/2)$. 
	
	The same argument shows that surjectivity on $\mathcal{X}^k(X_\delta)$ implies surjectivity on $\mathcal{X}^k(X_0)$, and indeed surjectivity on $\mathcal{X}^k(X_\delta)$ holds at $\lambda_0$ for $\Im \lambda_0 >0$ sufficiently large, as demonstrated in Lemma \ref{lem:qnm_def:L2invertible} and the proof of Theorem \ref{theo:mainextended} above. It then remains to show that $P(\lambda_0)$ is injective on $\mathcal{X}^k(X_0)$, and it suffices to do so for $k=0$. This follows as in the proof of Lemma \ref{lem:qnm_def:L2invertible}, using that the crucial Lemma \ref{lem:qnm_def:invertible1} also holds for $\delta =0$. The only subtlety involves the regularity of $u$ necessary for the integration by parts in Lemma \ref{lem:qnm_def:invertible1}; this can be handled by approximating $\phi u$ (in the notation of Lemma \ref{lem:qnm_def:L2invertible}) as in \cite[Lemma E.47]{zworski:resonances}

	Since the index of a continuous family of right semi-Fredholm operators is locally constant and it has just been shown that $P(\lambda_0)$ is invertible for $\Im \lambda_0>0$ sufficiently large, the proof of meromorphy is complete by analytic Fredholm theory. 
	
	The final step is to show that QNMs are in fact smooth up to $H_0$. In this paragraph $P(\lambda)$ acts on $\mathcal{X}^k(X_\delta)$ for a fixed $\delta \geq 0$. Near a pole $\lambda_0$ of $P(\lambda)^{-1}$ with $\Im \lambda_0 > -\varkappa(k+1/2)$, write
	\begin{align*}
	&P(\lambda) = P_{0} + (\lambda-\lambda_0)P_{1} + (\lambda - \lambda_0)^2P_2, \\
	&P(\lambda)^{-1} =  \sum_{j=1}^J(\lambda-\lambda_0)^{-j}A_{-j} +  A_0 + (\lambda-\lambda_0)H(\lambda), \noindent
	\end{align*}
	where $H(\lambda)$ is holomorphic near $\lambda_0$.
	Here the operators $A_{-j}$ are of finite rank for $j=1,\ldots,,J$. Analytic continuation gives the identities $P(\lambda)P(\lambda)^{-1} = 1$ and $P(\lambda)^{-1}P(\lambda) =1$, and hence
	\begin{equation}\begin{gathered} 
	P_0 A_{-J} = 0, \quad   A_0 P_0 +  A_{-1}P_1 = 1,  \\
	P_0 A_{-j} + P_1 A_{-j-1} = 0 \text{ for } j = 1,\ldots,J-1.
	\end{gathered} \label{eq:laurent}
	\end{equation}
	Restricting $A_0 P_0 +  A_{-1}P_1 = 1$ to the kernel of $P_0 = P(\lambda_0)$ shows that $\ker P(\lambda_0) \subseteq \ran A_{-1}$.

	It is now necessary to distinguish between $P(\lambda)$ on different spaces: temporarily write $P_\delta(\lambda)$ for $P(\lambda)$ acting on $\mathcal{X}^k(X_\delta)$, where $\delta \geq 0$. Fix $\delta >0$ and let 
	\[
	R: \mathcal{D}'(X_\delta) \rightarrow \mathcal{D}'(X_0)
	\]
	 be the restriction map. By analytic continuation from $\Im \lambda >0$ sufficiently large,
	\[
	P_0(\lambda)^{-1} = R \circ P_\delta(\lambda)^{-1} \circ E_k
	\]
	in the half-plane $\Im \lambda > -\varkappa(k+1/2)$. Therefore the residue $A_{0,-1}$ of $P_0(\lambda)^{-1}$ at a pole $\lambda_0$ with $\Im \lambda_0 > -\varkappa(k+1/2)$ is given by $R\circ A_{\delta,-1}\circ E_k$, where $A_{\delta,-1}$ is the residue of $P_\delta(\lambda)^{-1}$ at $\lambda_0$. On the other hand, if $u \in \mathcal{X}^k(X_\delta)$ satisfies $P_\delta(\lambda_0)u \in C^\infty(X_\delta \cup H_\delta)$, then $u \in C^\infty(X_\delta \cup H_\delta)$ by Proposition \ref{prop:regularity}. This observation combined with \eqref{eq:laurent} shows that the Laurent coefficients satisfy
	\begin{equation} \label{eq:smoothrange}
	\ran A_{\delta,-j} \subseteq C^\infty(X_\delta \cup H_\delta)
	\end{equation}
	for each $j = 1,\ldots,J$. In particular, $\ker P_0(\lambda_0) \subseteq \ran A_{0,-1} = \ran \left( R\circ A_{\delta,-1}\circ E_k\right)$, while $\ran \left( R\circ A_{\delta,-1}\circ E_k\right)  \subseteq C^\infty(X_0 \cup H_0)$ by \eqref{eq:smoothrange}, thus finishing the proof. \end{proof}

\section{Proof of Theorem \ref{theo:maintheo3}}

Given $m \in \mathbb{Z}$, define the space of distributions 
\[
\mathcal{D}'_m = \{ u \in \mathcal{D}' : (D_\phi - m)u = 0 \}.
\]
This definition applies to distributions on any of the spaces $\mathcal{M}_\delta, X_\delta$, or $Y$. Furthermore, if $T(\lambda)$ is axisymmetric, then one has the the mapping property 
\[
T(\lambda) : \mathcal{X}^k(X_\delta) \cap \mathcal{D}'_m(X_\delta) \rightarrow H^{k+1-\mu}(Y) \cap \mathcal{D}'_m(Y).
\]
Because it is assumed that $\Im \lambda > -\varkappa(k+1/2)$, it is enough to work with smooth functions (each of the kernels in Theorem \ref{theo:maintheo3} consists of smooth functions in that case). Given a fixed $\delta >0$, let
\[
X_- = X_\delta \setminus (X_0 \cup H_0),
\]
and $\OL{X}_- = X_- \cup H_\delta \cup H_0$ be its closure. 

\begin{prop} \label{prop:artificialwave}
	Let $m \in \mathbb{Z}$. Given $f \in C^\infty(X_\delta \cup H_\delta) \cap \mathcal{D}'_m(X_\delta)$ such that $\supp f \subset \OL{X}_-$, there exists a unique solution to the problem
	\[
	P(\lambda)u = f,\quad \supp u \subset \OL{X}_-,
	\]
	such that $u \in C^\infty(X_\delta \cup H_\delta) \cap \mathcal{D}'_m(X_\delta)$.
\end{prop}
Delaying the proof of Proposition \ref{prop:artificialwave} for a moment, Theorem \ref{theo:maintheo3} is now established by precisely the same argument as \cite[Lemma 2.2]{hintz2015asymptotics}:

\begin{proof} [Proof of Theorem \ref{theo:maintheo3}]  First suppose that $\nu \geq 1$. As in the proof of Theorem \ref{theo:maintheo1}, write $P_\delta(\lambda)$ for $P(\lambda)$ acting on $\mathcal{X}^k(X_\delta)$. For injectivity, take $v \in \ker P_\delta(\lambda) \cap \mathcal{D}'_m(X_\delta)$. If the restriction of $v$ to $X_0$ is zero, then $v$ is supported in $\OL{X}_-$, which implies that $v=0$ on $X_0$ according to Proposition \ref{prop:artificialwave}. For surjectivity, suppose that $u \in \ker P_0(\lambda) \cap \mathcal{D}'_m(X_0)$. Extend $u$ arbitrarily to $X_\delta$ as an element $\widetilde{u} \in C^\infty(X_\delta \cup H_\delta) \cap \mathcal{D}'_m(X_\delta)$; according to Proposition \ref{prop:artificialwave}, 
	the equation \[
	P_\delta(\lambda)v = P_\delta(\lambda)\widetilde{u}
	\]
	has a unique solution $v \in C^\infty(X_\delta \cup H_\delta) \cap \mathcal{D}'_m(X_\delta)$ such that $\supp v \subset \OL{X}_-$. Then $\widetilde{u}-v \in \ker P_\delta(\lambda) \cap \mathcal{D}'_m(X_\delta)$ and $\widetilde{u}-v$ restricts to $u$ on $X_0$. The same argument applies when $0 < \nu < 1$ since $T(\lambda)$ is axisymmetric, replacing $P(\lambda)$ with $\mathscr{P}(\lambda)$.
\end{proof}

Although Proposition \ref{prop:artificialwave} is closely related to the results of \cite{vasy2010wave} on asymptotically de Sitter spacetimes, a direct proof is outlined here --- see also \cite[Lemma 1]{zworski2015resonances} for the same type of result (at least for the uniqueness part).

Define the Riemannian metric 
\[
h = \frac{1}{\Delta_\theta} d\theta^2 + \frac{\Delta_\theta\sin^2\theta}{(1-\alpha)^2} d\phi^2,
\]
which extends smoothly across the poles to $\mathbb{S}^2$. The idea is to apply an energy identity in $X_-$.

Let $\rho = r_+ - r$, which is positive in $X_-$. Given $u \in C^\infty(X_\delta \cup H_\delta)$, let $d_y u$ denote the differential of $u(\rho,\cdot)$ on $\mathbb{S}^2$. Then for any $N \in \mathbb{R}$,
\begin{multline*}
\partial_\rho \left( \rho^N \left(-\Delta_r |\partial_\rho u|^2 + h^{-1}(d_y u, d_y u) + |u|^2 \right) \right)\\= 2\rho^N \Re \left(  \partial_\rho \bar u \, (\Delta_r D_\rho^2 u) + h^{-1}(d_y \partial_\rho u, d_y \bar u)    \right)\\ + N \rho^{N-1} \left(-\Delta_r |\partial_\rho u|^2 + |d_y u |_h^2 + |u|^2\right) +\rho^N R,
\end{multline*}
where $R$ is a smooth quadratic form in $(u,du)$ which is independent of $N$ (at this stage $R = -(\partial_r \Delta_r) |\partial_\rho u|^2 + 2 \Re \partial_\rho u \cdot \bar u$).  Given $0< \varepsilon < \rho \leq \delta$, integrate over the region $[\varepsilon, \rho] \times \mathbb{S}^2$ and apply Green's theorem to obtain
\begin{multline*}
\rho^N E(\rho) - \varepsilon^N E(\varepsilon) = 2 \int_{[\varepsilon, \rho] \times \mathbb{S}^2} \rho^N_1 \Re \left(\partial_\rho \bar u  \left( \Delta_r D_\rho^2 u + \Delta_h u \right)\right) d\rho_1 \,dh \\ + N\int_{\varepsilon}^{\rho}  \rho_1^{N-1} E(\rho_1) \, d\rho_1 + \int_{[\varepsilon, \rho] \times \mathbb{S}^2} \rho^N R\, d\rho_1\, dh,
\end{multline*}
where $\Delta_h$ is the nonnegative Laplacian for $h$ and
\[
E(\rho) = \int_{\mathbb{S}^2} \left(-\Delta_r |\partial_\rho u|^2 + h^{-1}(d_yu , d_y u) + |u|^2 \right)\, dh.
\]
In general, $\Delta_r D_\rho^2 + \Delta_h$ differs from $P(\lambda)$ by a second order operator. On the other hand, after restricting to $\mathcal{D}_m'(X_\delta)$ this difference is of first order and can be absorbed into $R$. Thus
\begin{multline} \label{eq:artificialenergy}
\rho^N E(\rho) - \varepsilon^N E(\varepsilon) = 2 \int_{[\varepsilon, \rho] \times \mathbb{S}^2} \rho^N_1 \Re\left( \partial_\rho \bar u \, P(\lambda)u \right) d\rho_1 \,dh \\  + N\int_{\varepsilon}^{\rho}  \rho^{N-1}_1 E(\rho_1) \, d\rho_1 + \int_{[\varepsilon, \rho] \times \mathbb{S}^2} \rho^N_1 R \, d\rho_1\, dh
\end{multline}
for each $u \in C^\infty(X_\delta \cup H_\delta)\cap \mathcal{D}'_m(X_\delta)$, where now $R$ is a quadratic form in $(u,du)$ which depends on $\lambda$ and $m$.

\begin{proof} [Proof of Proposition \ref{prop:artificialwave}]
	To prove uniqueness, suppose that $u \in C^\infty(X_\delta \cup H_\delta) \cap \mathcal{D}_m'(X_\delta)$ satisfies $P(\lambda)u=0$ and $\supp u \subset \OL{X}_-$. Observe that $u$ vanishes to infinite order at $H_0$, and therefore $\rho^N E(\rho) \rightarrow 0$ as $\rho \rightarrow 0$. Apply \eqref{eq:artificialenergy} with $N$ large and negative. Since $\Delta_r$ vanishes to first order at $\{\rho = 0\}$, there exists $N< 0$ such that
	\[
	N \rho_1^{N-1}E(\rho_1) + \rho_1^N \int_{\mathbb{S}^2}  R \, dh \leq 0
	\]
	for each $\rho_1 \in [0,\delta]$. Letting $\varepsilon \rightarrow 0$ shows that $E(\rho) = 0$ for each $\rho \in [0,\delta]$, hence $u=0$. 
	
	For the existence part of the proof, note that $\Delta_r D_\rho^2 + \Delta_h$ is formally self-adjoint with respect $d\rho \, dh$ modulo first order terms, so \eqref{eq:artificialenergy} also applies to $P(\lambda)^*$ computed with respect to $d\rho \, dh$, with a different error $R$ (observe that this adjoint is different than $P(\lambda)^*$ considered in Section \ref{subsect:kleingordon}).

	Assume that $v \in C^\infty(\OL X_-) \cap \mathcal{D}_m'(X_-)$ satisfies $\supp v \subset \{ \rho < \delta/2 \}$. In particular, $E(\delta) = 0$.
	Now take $N$ large and positive --- there exists $N >0$ and $C>1$ such that
	\[
	N \rho_1^{N-1}E(\rho_1) + \rho_1^N \int_{\mathbb{S}^2}  R \, dh \geq C^{-1} N \rho_1^{N-1}E(\rho_1)
	\]
	for $\rho_1 \in [0,\delta]$. Furthermore $\varepsilon^N E(\varepsilon) \rightarrow 0$ as $\varepsilon\rightarrow 0$ in light of the $\varepsilon^N$ factor.
Combined with Cauchy--Schwarz, this implies 
	\[
	N  \int_0^\delta  \rho^{N-1}\|  v(\rho,\cdot) \|_{H^1(\mathbb{S}^2)}^2 \,d\rho\leq C \int_0^\delta  \rho^{N} \| P(\lambda)^* v (\rho,\cdot)\|^2_{H^0(\mathbb{S}^2)} \, d\rho
	\]
	for $N > 0$ sufficiently large. Furthermore, by commuting with an axially symmetric elliptic pseudodifferential operator on $\mathbb{S}^2$ of negative order and absorbing the commutator into the left hand side by possibly increasing $N$,
	\begin{equation} \label{eq:artificialadjoint}
	N  \int_0^\delta  \rho^{N-1}\|  v(\rho,\cdot) \|_{H^{-s+1}(\mathbb{S}^2)}^2 \,d\rho\leq C \int_0^\delta  \rho^{N} \| P(\lambda)^* v (\rho,\cdot)\|^2_{H^{-s}(\mathbb{S}^2)}\, d\rho.
	\end{equation}
	Thus $N>0$ depends on $\lambda, m$, and $s$.
	
	Now suppose that $f \in C^\infty(\OL{X}_-) \cap \mathcal{D}_m'(X_-)$ vanishes to infinite order at $H_0$, so in particular 
	\[
	f \in \rho^{(N-1)/2} L^2((0,\delta); H^{s-1}(\mathbb{S}^2))\cap \mathcal{D}'_m(X_-)
	\]
	for each $N > 0$ and $s\in\mathbb{R}$. Define the form $\ell$ mapping
	\[
	\ell : P(\lambda)^*v \mapsto \left< f, v \right>_{L^2((0,\delta)\times\mathbb{S}^2)},
	\]
	where $v\in  C^\infty(\OL{X}_-) \cap \mathcal{D}_m'(X_-)$ and $\rho < \delta /2$ on the support of $v$. Then the estimate \eqref{eq:artificialadjoint} shows that $\ell$ is bounded on the set of all such $P(\lambda)^*v$ equipped with the $\rho^{-N/2}L^2((0,\delta);H^{-s}(\mathbb{S}^2))$ norm, provided $N>0$ is sufficiently large.
	
	By Hahn-Banach and the Riesz representation, there exists 
	\[
	u \in \rho^{N/2} L^2((0,\delta); H^s(\mathbb{S}^2)) \cap \mathcal{D}_m'(X_-)
	\] 
	such that
	\[
	\left< f, v \right>_{L^2((0,\delta)\times\mathbb{S}^2)} = \left< u, P(\lambda)^*v \right>
	\]
	for each $v$ as above, where the pairing on the right is duality between 
	\[
	\rho^{N/2} L^2((0,\delta); H^s(\mathbb{S}^2)) \cap \mathcal{D}_m'(X_-) \Longleftrightarrow \rho^{-N/2} L^2((0,\delta); H^{-s}(\mathbb{S}^2)) \cap \mathcal{D}_m'(X_-).
	\] 
	and $v$ is as above. In particular $P(\lambda)u = f$ in $\mathcal{D}'_m(\{0 < \rho < \delta/2\})$. Of course one can always choose an arbitrary smooth extension of $f$ from $\OL{X}_-$ up to $\rho = 2\delta$ and then run the previous argument with $\delta$ replaced by $2\delta$, thus obtaining a distributional solution on all of $X_-$.
		
	Once $s>0$ and $N>0$ are sufficiently large, Sobolev regularity of $u$ in the $\rho$ variable follows from the usual ``partial hypoellipticity at the boundary" argument (using the high order of vanishing of $u$ and $f$ to account for the derivatives in the $\rho$ variable which degenerate at $H_0$), see \cite[Theorem B.2.9]{hormanderIII:1985}. Given sufficient regularity and order of vanishing, the solution $u$ is unique by the energy estimates for $P(\lambda)$; thus there exists a solution $u$ which is smooth on $\OL{X}_-$ and vanishes to infinite order at $H_0$.
\end{proof}

\section*{Acknowledgements}
I would like to thank Maciej Zworski for his encouragement and numerous helpful conversations. I am especially grateful to Semyon Dyatlov and Peter Hintz for illuminating discussions on the microlocal analysis of black hole spacetimes and for their continued interest in the problem. Finally, I would like to thank Andr\'as Vasy for clarifying the applicability of his results in \cite{vasy:2013} to Kerr--AdS metrics, and the anonymous referee for some useful comments.

\bibliographystyle{plain}

\bibliography{biblio.bib}

\end{document}